\documentclass[oneside]{amsart}

\usepackage[T1]{fontenc}

\usepackage[latin1]{inputenc}

\usepackage{amsmath,amsthm,amsfonts,amssymb,latexsym, mathrsfs}

\makeatletter

\theoremstyle{plain}
\newtheorem{thm}{Theorem}[section]

\numberwithin{equation}{section} %% Comment out for sequentially-numbered

\numberwithin{figure}{section} %% Comment out for sequentially-numbered

\theoremstyle{plain}
\newtheorem{cor}[thm]{Corollary} %%Delete [thm] to re-start numbering

\theoremstyle{definition}
\newtheorem{defn}[thm]{Definition}

\theoremstyle{plain}
\newtheorem{lem}[thm]{Lemma} %%Delete [thm] to re-start numbering

\theoremstyle{plain}
\newtheorem{prop}[thm]{Proposition} %%Delete [thm] to re-start numbering

\theoremstyle{plain}
\newtheorem{fact}[thm]{Fact}

\theoremstyle{plain}
\newtheorem{notation}[thm]{Notation}

\newtheorem{rem}[thm]{Remark}

\theoremstyle{plain}

\newtheorem{rema}[thm]{Remark}

\ifx\pdfoutput\@undefined\usepackage[usenames,dvips]{color}

\else\usepackage[usenames,dvipsnames]{color}

\IfFileExists{pdfcolmk.sty}{\usepackage{pdfcolmk}}{}

\fi

\makeatletter

\usepackage{amssymb}

\usepackage{amsmath}

\newcommand{\cL}{\mathcal{L}}

\def \<{\langle}
\def \>{\rangle}

\usepackage{amsfonts}

\usepackage{amscd}

\usepackage{latexsym}

\usepackage{epsfig}

\makeatother

\def\Ind#1#2{#1\setbox0=\hbox{$#1x$}\kern\wd0\hbox to 0pt{\hss$#1\mid$\hss}
\lower.9\ht0\hbox to 0pt{\hss$#1\smile$\hss}\kern\wd0}
\def\ind{\mathop{\mathpalette\Ind{}}}
\def\Notind#1#2{#1\setbox0=\hbox{$#1x$}\kern\wd0\hbox to 0pt{\mathchardef
\nn=12854\hss$#1\nn$\kern1.4\wd0\hss}\hbox to
0pt{\hss$#1\mid$\hss}\lower.9\ht0 \hbox to
0pt{\hss$#1\smile$\hss}\kern\wd0}
\def\nind{\mathop{\mathpalette\Notind{}}}

\newcommand{\acl}{\operatorname{acl}}
\newcommand{\acle}{\operatorname{acl^{eq}}}

\newcommand{\tp}{\operatorname{tp}}
\newcommand{\stp}{\operatorname{stp}}
\newcommand{\etp}{\operatorname{etp}}

\newcommand{\dcl}{\operatorname{dcl}}

\newcommand{\cl}{\operatorname{cl}}
\newcommand{\ccl}{\operatorname{cl}}

\newcommand{\scl}{\operatorname{scl}}

\begin{document}

\title{Supersimple structures with a dense independent
subset}

\author{Alexander Berenstein}
\address{Universidad de los Andes,
Cra 1 No 18A-10, Bogot\'{a}, Colombia}
 \urladdr{www.matematicas.uniandes.edu.co/\textasciitilde aberenst}

\author{Juan Felipe Carmona}
\address{Universidad Antonio Nari\~no, Calle 58 A No. 37 - 94, Bogot\'{a}, Colombia}
\email{jfcarmonag@gmail.com}

\author{Evgueni Vassiliev}
\address{Grenfell Campus, Memorial University of Newfoundland, Corner Brook, NL A2H 6P9, Canada}
\email{yvasilyev@grenfell.mun.ca}

\thanks{The first author was partially supported by the project of Colciencias "M\'etodos de estabilidad en clases no estables". The authors would like to thank Daniel Palac\'in for some helpful coversations and the referee for some valuable suggestions}

\date{}

\begin{abstract}
Based on the work done in \cite{BV-Tind,DMS} in the o-minimal and geometric settings, we study expansions of models of a supersimple theory with a new predicate distiguishing a set of forking-independent elements that is dense inside a partial type $\mathcal{G}(x)$, which we call $H$-structures. We show that any two such expansions have the same theory and that under some technical conditions, the saturated models of this common theory are again $H$-structures. We prove that under these assumptions the expansion is supersimple and characterize forking and canonical bases of types in the expansion. We also analyze the  effect these expansions have on one-basedness and CM-triviality. In the one-based case, 
when $T$ has $SU$-rank $\omega^\alpha$ and the $SU$-rank is continuous, we take $\mathcal{G}(x)$ to be the type of elements of $SU$-rank $\omega^\alpha$ and we describe a natural "geometry of generics modulo $H$" associated with such expansions and show it is modular.
\end{abstract}

\pagestyle{plain}

\maketitle

\section{Introduction}

There are several papers that deal with expansions of simple theories  with a new unary predicate. For example, there is the expansion with a random subset \cite{CP} that gives a case where the new theory is again simple and forking remains the same, in contrast to  the case of lovely pairs \cite{BPV,Va1}, where the pair is usually much richer and  the complexity of forking is related to the geometric properties of the underlying theory \cite{Va1}.

In \cite{BV-Tind} the first and the third authors studied, in the setting of geometric structures, adding a predicate for an algebraically independent set $H$ which is dense and codense in a model $M$ (meaning every non-algebraic formula in a single variable has a realization in $H$ and a realization not algebraic over $H$ and its parameters). Such expansions are called $H$-structures and came as a generalization of ideas developed in the framework of o-minimal theories in \cite{DMS}.  The key tool used in \cite{BV-Tind}  was that the closure operator $\acl$ has the exchange property and thus gives a matroid that interacts well with the definable subsets. A special case under consideration was SU-rank one theories, where forking independence agrees with algebraic independence. In the $SU$-rank one setting the 
authors characterized forking in the expansion and gave a description of canonical bases. As in the lovely pair case, the complexity of forking is related to the underlying geometry of the base theory $T$.

In this paper we bring together ideas from lovely pairs of simple theories and $H$-structures to the setting of supersimple theories, 
we use as a main tool forking independence. We fix a complete supersimple theory and a type $\mathcal{G}(x)$ and consider the expansion where we name as a predicate realizations of  $\mathcal{G}(x)$ that are forking independent and that satisfy a density property with respect to $\mathcal{G}(x)$, namely every formula that belongs to a non-forking extension of $\mathcal{G}(x)$ has a realization in the predicate. We also assume that the structure is rich 
with respect to the predicate: every formula $\varphi(x,\vec a)$ has a realization that is forking independent from the predicate over $\vec a$. 

We call such expansions $H$-structures associated to $\mathcal{G}$. We first prove that all such structures define the same theory, which we call $T^{ind}_\mathcal{G}$. When a $|T|^+$-saturated model of $T^{ind}_\mathcal{G}$ is again an $H$-structure associated to $\mathcal{G}$ we say that the class of $H$-structures associated to $\mathcal{G}$ is first order. We show that the class is first order under a weakening of wnfcp and under some definability condition on non-forking extensions of $\mathcal{G}(x)$. When the class if first order, we prove that the 
extension is supersimple and characterize forking.  In particular,  we get a clear description of canonical bases in the expansion, up to interalgebraicity (see Proposition \ref{cbTind}).

Assume now that $T$ is a  one-based supersimple theory of SU-rank $\omega^\alpha$ and that the $SU$-rank is continuous. 
Take $\mathcal{G}(x)$ to be the (partial) type of rank $\omega^\alpha$, which can be seen as the generic elements in $T$. 
Then there is a natural pregeometry associated to the generics, namely 
$a\in \cl(B)$ if $SU(a/B)<\omega^\alpha$. We then use this expansion to study the underlying geometry of the closure operator localized in $H$.
We show that if $T$ is a  one-based supersimple theory of SU-rank $\omega^\alpha$,
$(N,H)$ a sufficiently (e.g. $|T|^+$-) saturated $H$-structure, then the localized closure operator $\cl(-\cup H)$ is modular and its associated geometry is a disjoint union of projective geometries over division rings and trivial geometries. These closures have been studied also for lovely pair like constructions, see for example the work of Fornasiero on lovely pairs of closure operators \cite{Fo}.

Of special interest is the effect of our expansion on the geometric complexity, namely the ampleness hierarchy. Following the ideas of \cite{Ca}, we show that the expansion preserves CM-triviality. When $T$ is one-based, $\mathcal{G}$ is $x=x$, then $T^{ind}_\mathcal{G}$ is one-based if and only if forking is trivial.
 
When $T$ is one-based of SU-rank $\omega^\alpha$ with continuous rank and we choose $\mathcal{G}$ to be the partial type of SU-rank $\omega^\alpha$, then, as described above, the theory induces a closure operator: $a\in \cl(B)$ if $SU(a/B)<\omega^\alpha$. In this setting, $T^{ind}_\mathcal{G}$ is one-based if and only if the pregeometry associated to $\cl$ is trivial.

Finally we mention another paper related to this expansion. Assume the theory $T$ is superstable and $\mathcal{G}(x)$ is a
stationary type over $\emptyset$, then the interpretation of the predicate is a Morley sequence; this case is related to the 
work done in \cite{BB}.

This paper is organized as follows. In section \ref{sec:H-structures} we start with a complete supersimple theory $T$ and and a type-definable set $\mathcal{G}$ and we define the class of $H$-structures associated to a theory $\mathcal{G}$. We show that two such $H$-structures are elementarily equivalent and call $T^{ind}_\mathcal{G}$ this common theory. Finally we prove that under some technical conditions (elimination of the quantifier $\exists^{large}$ and the
type definability of the predicates $Q_{\varphi,\psi}$) the saturated models of $T^{ind}_\mathcal{G}$ are again $H$-structures.

In section \ref{sec:examples} we study several examples of supersimple theories: $SU$-rank one structures, differentially closed fields,
the free pseudoplane, vector spaces with a generic automorphism, $\omega$-stable theories with definable Morley rank, $H$-pairs and lovely pairs of geometric theories. In most cases we show the corresponding theory of $H$-structures with a reasonable choice for $\mathcal{G}$
is first order. In the cases of ACFA and a vector space with a generic automorphism, when $\mathcal{G}$ is the generic type, being first order is reduced to type definability of 
the predicates $Q_{\varphi,\psi}$.

In section \ref{sec:definablesets} we analyze the definable sets in the expansion, we prove that every definable set is a boolean combination of old formulas bounded by existential quantifiers over the new predicate.  In section \ref{sec:supersimplicity}
we characterize forking in the expansion and characterize canonical bases. In section \ref{1based-section}  we study the question of preservation of one-basedness and then in section \ref{Ampleness-section} the preservation of CM-triviality under our expansion. Finally in section \ref{geometrymodH} we consider the special case where $T$ is a one-based supersimple theory of SU-rank $\omega^\alpha$ such that the $SU$-rank is continuous and take $\mathcal{G}(x)$ to be the (partial) type of rank $\omega^\alpha$ . We study the geometry of $\cl(-\cup H)$, where $\cl$ is the closure associated to the generic elements in the theory.

\section{$H$-structures: definition and first properties}\label{sec:H-structures}

Let $T$ be a supersimple theory in a language $\cL$, let $M\models T$ be a monster model. 
Since our only assumption on $T$ is supersimplicity, if needed, we can assume $T=T^{eq}$.
Fix a partial
1-type $\mathcal{G}(x)$ over $\emptyset$.

For $M\models T$, $A\subset M$ and $b\in M$, we say that that $b$ is \emph{large} over $A$ (or that
$\tp(b/A)$ is large) if $\tp(b/A)$ is a non-forking  extension of $\mathcal{G}$. Otherwise we say
that $b$ is \emph{small} over $A$ (or that $\tp(b/A)$ is small).

We will say a formula $\varphi(x,\vec a)$ is \emph{large} if there is $b\models \varphi(x,\vec a)$
such that $b$ is large over $\vec a$. We will say a formula $\varphi(x,\vec a)$ is \emph{small} if 
it is not large. That is, for all $b\models \varphi(x,\vec a)$, $\tp(b/\vec a)$ either does not 
extend $\mathcal{G}$ or it is a forking extension of $\mathcal{G}$.

The goal of this paper is to merge the ideas on $H$-structures developed in \cite{BeVa} 
with those of lovely pairs of simple theories \cite{BPV} and develop the notion of $H$-structures
relative to the family $\mathcal{G}$.

Following the ideas from \cite{BeVa}, we let $H$ be a new unary predicate that does not appear in $\cL$ and define $\cL_H=\cL\cup \{H\}$. So an $\cL_H$-structure is a pair $(M,H(M))$, where $H(M)\subset M$.

\begin{notation}
Let $(M,H(M))$ be an $\cL_H$-structure and let $A\subset M$. We write $H(A)$ for
$H(M)\cap A$.
\end{notation}

\begin{notation}
Throughout this paper \emph{independence} means forking
independence in the sense of $T$ and we use the familiar symbol $\ind$.
We write $\tp(\vec a)$ for the $\cL$-type
of $a$ and $\dcl$, $\acl$ for the definable closure and the algebraic
closure in the language $\cL$.
Similarly we write $\dcl_{H}, \acl_{H}, \tp_{H}$ for the definable closure,
the algebraic closure and the type in the language $\cL_H$.
\end{notation}

\begin{defn}\label{Hstructures}
We say that $(M,H(M))$ is an $H$-\emph{structure associated to $\mathcal{G}$} if
\begin{enumerate}
\item For every $h\in H(M)$, $h \models \mathcal{G}$.
\item For all $n\geq 1$, if $h_1,\dots,h_n\in H(M)$ are distinct, then they are independent.
\item (Density/coheir property) If $A\subset
M$ is finite and $q\in
S_1(A)$ is the type of a non-forking extension of $\mathcal{G}$, there is  $a\in H(M)$ such that
$a\models q$.
\item (Co-density/extension property)  If $A\subset M$ is
finite and $q\in S_1(A)$, there is  $a\in M$, $a\models q$ and $a\ind_A H(M)$.
\end{enumerate}
\end{defn}

\begin{lem}
The pair $(M,H(M))$ is an $H$-structure associated to $\mathcal{G}$ if and only if:
\begin{enumerate}
\item $H(M) \models \mathcal{G}$.
\item For all $n\geq 1$, if $h_1,\dots,h_n\in H(M)$ are distinct, then they are independent.
\item[(2')] (Generalized density/coheir property) If $A\subset M$ is
finite and $q\in S_n(A)$ is the type of  an independent $n$-tuple of 
non-forking extensions of $\mathcal{G}$, then
there is $\vec a\in H(M)^n$ such that $\vec a\models q$.
\item[(3')] (Generalized co-density/extension property)  If $A\subset M$ is
finite and $q\in S_n(A)$, then there is $\vec a\in M^n$
realizing $q$ such that $\tp(\vec a/A\cup H(M))$ does not fork over $A$.
\end{enumerate}
\end{lem}

\begin{proof}
We prove (2') and leave (3') to the reader. Let $\vec b\models q$,
we may write $\vec b=(b_1,\dots,b_n)$.
 Since $(M,H(M))$ is an $H$-structure, applying the density property we can
 find $a_1 \in H(M)$ such that $\tp(a_1/A)=\tp(b_1/A)$. Let
 $q(x,b_1,A)=tp(b_2,b_1,A)$ and let $A_1=A\cup\{a_1\}$. Now consider the type $q(x,a_1,A)$
 over $A_1$, which is some non-forking extension of $\mathcal{G}$. Applying the density property we can
 find $a_2 \in H(M)$ such that $\tp(a_2,a_1/A)=\tp(b_2,b_1/A)$. We continue inductively
 to find the desired tuple $(a_1,a_2,\dots,a_n)$.
\end{proof}

Note that if $(M,H(M))$ is an $H$-structure, the extension property
implies that $M$ is $\aleph_0$-saturated. Also note that Definition \ref{Hstructures} can be generalized to the setting of simple theories following the ideas of \cite[Definition 3.1]{BPV}. Finally note that if the codensity property holds for types in the real sort, then it also holds for types in imaginary sorts.

\begin{defn}
Let $A$ be a subset of an $H$-structure $(M,H(M))$.
We say that $A$ is $H$-\emph{independent} if $A$ is independent
from $H(M)$ over $H(A)$.
\end{defn}

\begin{lem}\label{existence}
Any model $M$ of $T$ with a distinguished independent subset
$H(M)$ can be embedded in an $H$-structure in an $H$-independent
way.
\end{lem}

\begin{proof}
Given any model $M$ with a distinguished independent subset $H(M)$ of elements realizing 
$\mathcal{G}$, we can always
find an elementary extension $N$ of $M$ and a set $H(N)$ of independent realizations of $\mathcal{G}$ extending $H(M)$ such
that for every generic $1$-type $p(x,\acl(\vec m))$ (i.e. $p(x,\acl(\vec m))$ is a non-forking extension of $\mathcal{G}$), where $\vec m \in M$, there is $d\in H(N)$ such that $d\models p(x,\acl(\vec m))$ and $d\ind_{H(M)} \vec m$. Add a similar
statement for the extension property. Now apply a chain argument.
\end{proof}

In particular,  $H$-structures exist. The following is the main result from this section:

\begin{prop}\label{H-indtypes}
Let $(M,H)$ and $(N,H)$ be sufficiently saturated $H$-structures,
$\vec a\in M$ and $\vec a'\in N$ $H$-independent tuples such that
$\tp(\vec a, H(\vec a))=\tp(\vec a', H(\vec a'))$.  Then
$tp_H(\vec a)=\tp_H(\vec a')$.
 \end{prop}

 \begin{proof} Write $\vec a=\vec a_0 \vec a_1 \vec h$, where $\vec a_0$ is
independent over $H(M)$, $\vec h=H(\vec a)\in H(M)$ and $\vec a_1$ is small over $\vec
a_0\vec h$. Similarly write $\vec a'=\vec a_0' \vec a_1'\vec h'$.

It suffices to show that for any $b\in M$ there are $\vec h_1\in
H(M)$, $\vec h_1'\in H(N)$ and $b'\in N$ such that $\vec a\vec h_1
b$ and $\vec a'\vec h_1' b'$ are each $H$-independent, $\tp(\vec
a_0 \vec a_1 \vec h \vec h_1b)=\tp(\vec a_0' \vec a_1' \vec h'
\vec h_1'b')$, and $b\in H(M)$ iff $b'\in H(N)$.

\noindent

Case 1: $b\in H(M)$, $b$ is small over $\vec a$. Then $\tp(b/\vec a)$ forks over $\emptyset$. Since $\vec a$ is $H$-independent, we must have $b\ind_{\vec h} \vec h \vec a$ and by transitivity  $b\nind \vec h$. Since $H$ forms an independent set we must have $b\in \vec h$. Let $b'\in \vec h'$
be such that  $\tp(b'\vec a')=\tp(b\vec a)$ and the result follows.
Here we can take $\vec h_1$ and $\vec h_1'$ to be empty.

\noindent

Case 2: $b\in H(M)$ and is large over $\vec a$. Then
$tp(b/\vec a)$ is generic, that is, it is a non-forking extension of $\mathcal{G}$. 
Let $f:M\to M'$ be a partial elementary map sending $\vec a$ to $\vec a'$.
Then $p(x,\vec a')=f(tp(b/\vec a))$ is a non-forking extension of $\mathcal{G}$. 
By the density property, we can find $b'\in H(N)$ such that $b'\models p(x,\vec a')$. That is, $\tp(b'\vec
a')=\tp(b\vec a)$. Note that both tuples $b\vec a$, $b'\vec a'$ are $H$-independent. Here again we can take $\vec h_1$ and $\vec
h_1'$ to be empty.

\noindent

Case 3: $b\not \in H(M)$. By supersimplicity, there is a finite tuple $\vec h_1\in H(M)$ such that $b \ind_{\vec a \vec h_1} H(M)$. We may assume that $\vec h_1$ is disjoint from $\vec h$
and thus independent from $\vec h$. Now observe that since $\vec a$ is $H$-independent, so 
is $\vec a \vec h_1$. By Case 2, there is $\vec h_1'\in H(N)$ such that $\tp(\vec a \vec h_1)=\tp(\vec a' \vec h_1')$. Let $p(x,\vec a \vec h_1)=\tp(b/\vec a \vec h_1)$. 
Now use the extension property to find $b'\in N$ such that $b'\models p(x,\vec a',\vec h_1')$, $b'\ind_{\vec a' \vec h_1'} H(N)$. Then by transitivity $\vec a' \vec h_1'b'$ is $H$-independent.

\end{proof}

The previous result has the following consequence:

\begin{cor}\label{completetheory}
All $H$-structures associated to $\mathcal{G}(x)$ are elementarily equivalent.
\end{cor}

We write $T^{ind}_{\mathcal{G}}$ for the common complete theory of all $H$-structures
of models of $T$.

\begin{defn}
We say that $T^{ind}_{\mathcal{G}}$ is \emph{first order} if the $|T|^+$-saturated models of $T^{ind}_{\mathcal{G}}$
are again $H$-structures associated to $\mathcal{G}$.
\end{defn}

To axiomatize $T^{ind}_{\mathcal{G}}$ and to show that $T^{ind}_{\mathcal{G}}$ is first order, we follow the ideas of \cite[Prop 2.15]{Va1}, \cite{BeVa}
and \cite{BPV}. We will need several technical assumptions for this result

\begin{defn} We say that being large is \emph{definable} if
for every formula $\varphi(x,\vec y)$ there is a formula $\psi(\vec y)$ such that for any $\vec a\in M$, $\varphi(x,\vec a)$ is large if and only if $\psi(\vec a)$. When this is the case, we write $\exists^{large}x \varphi(x,\vec y)$ if $\psi(\vec y)$ holds.
\end{defn}

We also need the following definition from \cite[Definition 2.4]{BPV}:

\begin{defn} Let $\psi(\vec y, \vec z)$ and $\varphi(\vec x, \vec y)$ be $\cL$-formulas. $Q_{\varphi,\psi}$
is the predicate which is defined to hold of a tuple $\vec c$ (in $M$) if for all $\vec b$ satisfying
$\psi(\vec y, \vec c)$, the formula $\varphi(\vec x, \vec b)$ does not divide over $\vec c$.
\end{defn}

The following result follows word by word from the proof of \cite[Proposition 4.5]{BPV}, changing the elementary
substructure for the predicate $H$:

\begin{prop}\label{axiomextension}
The following are equivalent:
\begin{enumerate}
\item $Q_{\varphi,\psi}$ is type-definable (in M ) for all $\cL$-formulas $\varphi(\vec x,\vec y)$,$\psi(\vec y,\vec z)$.
%\item The extension property is first order.
\item Any $|T|^+$-saturated model of $T^{ind}_{\mathcal{G}}$ satisfies the extension property.
\end{enumerate}
\end{prop}

In \cite{BPV} it has been proved that, whenever $T$ is a simple theory, its associated theory of lovely pairs $T_P$ is axiomatizable if and only if $T$ is low and the predicates $Q_{\varphi,\psi}$ are type definable. In this case we say that $T$ has the {\it weak non-finite cover property (wnfcp)} (For $T$ stable, it is known that $T_P$ is axiomatizable if and only if $T$ has the nfcp, therefore this definition is rather natural).

\begin{cor}\label{wnfcpaxiom}
Let $T$ be a supersimple theory that satisfies wnfcp. Then the
extension property is first order.
\end{cor}

Finally we prove that under some strong assumptions, $T^{ind}_{\mathcal{G}}$ is first order:

\begin{prop}\label{axioms}
Assume $T$ eliminates $\exists ^{large}$ and that the predicates $Q_{\varphi,\psi}$ are $\cL$-type-definable
for all $\cL$-formulas $\varphi(\vec x,\vec y)$,$\psi(\vec y,\vec z)$. Then, if $(M,H)\models T^{ind}_{\mathcal{G}}$ is 
$|T|^+$-saturated then it is an $H$-structure associated to $\mathcal{G}$.
\end{prop}

\begin{proof}
We start by showing that all elements in $H$ satisfy the partial type $\mathcal G$. Let $\theta(x)\in \mathcal{G}$, then $\forall x ( H(x) \to \theta(x))\in T^{ind}_{\mathcal{G}}$. Thus for $h\in H(M)$, $M\models \theta(h)$
and we have $H(M)\subset \mathcal{G}(M)$.

Let us prove that $(M,H)$ satisfies the density property.
Let $p(x,\vec a)$ be non-forking extension of $\mathcal{G}$. Let
$\varphi(x,\vec a)\in p(x,\vec a)$, so $\varphi(x,\vec a)$ is large. Then 
$M\models \exists^{large}x \varphi(x,\vec a)$ and  all $H$-structures
$(N,H)$ satisfy $$\forall \vec y (\exists^{large}x\varphi(x,\vec y)\to \exists h(H(h)\wedge \varphi(h,\vec y)).$$ Thus there is $h\in H(M)$
such that $M\models \varphi(h,\vec a)$. Since $M$ is $|T|^+$-saturated,
there is $h\in H(M)$ such that $M\models p(h,\vec a)$.

Now we prove that the elements in $H$ are independent. Notice that the following sentence is true in every $H$-structure: $$\forall h_0,h_1,...,h_n\in H((\bigwedge h_i\neq h_j)\wedge \theta(h_0,h_1,...,h_n)\to \exists^{large}x\theta(x,h_1,...,h_n)).$$ Therefore if $d,c_1,\dots,c_n\in H(M)$ are all distinct
and $\theta(x,c_1,\dots,c_n)\in tp(d/c_1,\dots,c_n)$ we have 
$\exists^{large}x \theta(x,c_1,\dots,c_n)$ and in particular $\theta(x,c_1,\dots,c_n)$
does not fork over $\emptyset$. Thus all formulas in the type $tp(d/c_1,\dots,c_n)$ 
do not fork over $\emptyset$ and thus $d\ind \{c_1,\dots,c_n\}$.

Finally by Proposition \ref{axiomextension} any $|T|^+$-saturated model of $T^{ind}_{\mathcal{G}}$ satisfies
the extension property.
\end{proof}

\section{Examples}\label{sec:examples}
In this section we give a list of examples of supersimple theories with a type definable set $\mathcal{G}$ that eliminate
$\exists^{large}$ and where the extension property is first order. We also list some examples that
eliminate the quantifier $\exists^{large}$ but where it remains as an open question if the
extension property is first order.

\subsection{Theories of $SU$-rank one}
In this example we show that the results from \cite{BeVa} are a special case of our new setting. Our presentation 
will be sketchy, for more details see \cite{BeVa}.

Let $T$ be a complete theory of $SU$-rank one and let $\mathcal{G}$ be the intersection of all the 
non-algebraic types (i.e. the set of all formulas whose negation is algebraic). Then $T$ is geometric, i.e. the algebraic closure satisfies the exchange property and $T$ eliminates $\exists^{\infty}$.  In this setting a formula is large if it is non-algebraic and it is small if it is algebraic. Since $T$ eliminates $\exists^\infty$ then $T$ eliminates $\exists^{large}$. Finally since independence is captured by algebraic independence, then the extension property is first order.

\subsection{Differentially closed fields}
Let $T=DCF_0$, the theory of differentially closed fields. This theory is stable of
$U$-rank $\omega$ and also $RM(DCF_0)=\omega$.

First we study the extension property. Recall that $DCF_0$
has quantifier elimination \cite[Theorem 2.4]{Ma} and eliminates imaginaries \cite[Theorem 3.7]{Ma}.
It is proved in \cite[Theorem 2.13]{Ma} that $DCF_0$ has uniform bounding (i.e. it eliminates $\exists^\infty$)
and thus it has nfcp. This is also explicitly explained in \cite[page 52]{Ma}. It follows by
Corollary \ref{wnfcpaxiom} that the extension property is first order.

Let $p(x)$ be the unique generic type of the theory, that is, the type of a differential trascendental. This type
is complete, stationary and definable over $\emptyset$. Choose 
$\mathcal{G}$ to be the type $p(x)$. Let $\varphi(x,\vec y)$ be a formula
and let $\psi(\vec y)$ be its $p$-definition. Then for
$(K,d)\models DCF_0$, $\vec a\in K$, the formula $\varphi(x,\vec a)$ is
large iff $\psi(\vec a)$. Thus this theory eliminates the quantifier
$\exists^{large}$.

Also note that we can choose other partial types for $\mathcal{G}$. For example, 
let $\mathcal{G}$ be the non-algebraic constants, given by the type-definable set 
$\delta(x)=0$, $x\not \in \acl(\emptyset)$. It is proved in \cite[Lemma 5.6]{Ma} that the only
definable subsets of the field of constants are definable in the field language within the field of constants, thus 
the constants are stably embedded and strongly minimal. The partial type 
$\mathcal{G}$ corresponds to the generic type of $\delta(x)=0$, which is stationary 
and definable  over $\emptyset$. As before, the definability of the type implies elimination 
of $\exists^{large}$.

Since beautifull pairs associated to nfcp theories are again nfcp (see \cite{Po}), we can
also work on a lovely pair of differentially closed fields (which has $SU$-rank $\omega^2$) and 
choose $\mathcal{G}$ to be some stationary type in the pair definable over $\emptyset$, for example
the type of the generic elements in the pair.

\subsection{Free pseudoplane-infinite branching tree} \label{freeps}
Let $T$ be the theory of the free pseudoplane, that is, a graph without cycles such that every vertex has infinitely many edges. 
The theory of the free pseudoplane is stable of $U$-rank $\omega$ and $MR(T)=\omega$. For every $A$, $\acl(A)=\dcl(A)=A\cup\{x|\mbox{ there are points } a,b\in A \mbox{ and a path connecting them passing trough } x\}$. For $A$ algebraically closed and $a$ a single element, $U(a/A)= d(a,A)$ where $d(a,A)$ is the minimum length of a path from $a$ to an element of $A$ or $\omega$ if there is no path; in this last case we say that $a$ is at \emph{infinite distance to $A$} or that $a$ is \emph{not connected to $A$}. 
Note that there is a unique generic type over $A$, namely the type of an element which is not connected to $A$.  Choose $\mathcal{G}(x)$ to be this type. The generic type is definable over $\emptyset$ and thus by definability of types $T$ eliminates the quantifier $\exists^{large}$. 

An $H$-structure $(M,H)$ associated to $\mathcal{G}$ is an infinite collection of trees with an infinite collection of selected points $H(M)$ at infinite distance one from the other and with infinite many trees not connected to them. If $(N,H)\models Th(M,H)$, then $N$ has infinitely many selected points $H(N)$ at infinite distance one from the other. 

If $(N,H)$ is $\aleph_0$-saturated, then by saturation it also has infinitely many trees which are not connected to the points $H(N)$. We will prove that in this case $(N,H)$ is an $H$-structure associated to $\mathcal{G}(x)$. Since $T$ eliminates the quantifier $\exists^{large}$,  the density 
property holds as well as the fact that the elements in $H(N)$ are independent (both these properties are easy to check directly too).
Now let $A\subset N$ be finite and assume that $A=\dcl(A)$ and let $c\in N$. If $U(c/A)=\omega$ choose
a point $b$ in a tree not connected to $A\cup H$, then $\tp(c/A)=\tp(b/A)$ and $b \ind_A H$. If $U(c/A)=0$ there is nothing to prove. If $U(c/A)=n>0$, let $a$ be the nearest point from $A$ to $c$. Since there is at most one point of $H$ connected to $a$ and the trees are infinitely branching, we can choose a point $b$ with $d(b,a)=n$ and such that $d(b,A\cup H)= n$; then $\tp(c/A)=\tp(b/A)$ and $b \ind_A H$. This proves that $(N,H)$ is
an $H$-structure and that that $T^{ind}_\mathcal{G}$ is first order.

\subsection{Theories of Morley rank omega with definable Morley rank}\label{omegacase}
Let $T$ be a $\omega$-stable theory of rank $\omega$ and let $M\models T $ be $|T|^+$-saturated. Assume 
also that the Morley rank is definable, that is, for every formula $\varphi(x,\vec y)$ without parameters and every $\alpha\in \{0,1,\dots,\omega\}$ there is a formula $\psi_\alpha(\vec y)$ without parameters such that for $\vec a\in M$, $MR(\varphi(x,\vec a))\geq \alpha$ if and only if $\psi_\alpha(\vec a)$. To simplify the notation, we will write $MR(\varphi(x,\vec a))\geq \alpha$ instead of $\psi_\alpha(\vec a)$. Let $\mathcal{G}$ be the intersection of the types of Morley rank $\omega$, i.e. the formulas whose complement have rank $<\omega$.
We will prove that $T^{ind}_{\mathcal{G}}$ is first order.

Elimination of $\exists^{large}$: Consider first $\varphi(x,\vec y)$ and let $\vec b\in M$. Then $\varphi(x,\vec b)$ is large if and only if $MR(\varphi(x,\vec b))\geq \omega$, so $T$ eliminates the quantifier $\exists^{large}$. 

Extension property: Assume that $(M,H)$ is an $H$-structure associated to $\mathcal{G}$ and let $(N,H)\models Th(M,H)$ be $|T|^+$-saturated.
Let $a\in N$ and let $\vec b\in N$. If $MR(\tp(a/\vec b))=0$ there is nothing to prove. Assume then 
that $MR(\tp(a/\vec b))=n>0$.

Let $\varphi(x,\vec y)\in \tp(a,\vec b)$ with $MR(\varphi(x,\vec b))=n$ and $Md(\varphi(x,\vec b))=Md(\tp(a/\vec b))$. Let $(a',\vec b')\models tp(a,\vec b)$ belong to $M$. Since $(M,H)$ is an $H$-structure, we may assume that $a'\ind_{\vec b'} H$ and thus for every formula 
$\theta(x,\vec y,\vec z)$ and every tuple $\vec h\in H$, if $MR(\theta(x,\vec b',\vec h))<MR(\varphi(x,\vec b'))=n$
then $\neg \theta(x,\vec b',\vec h)\in \tp(a'/\vec b'H)$. So $(M,H)\models \forall d' (MR(\varphi(x,\vec d'))\geq n \to \exists c \varphi(c,\vec d')\wedge 
\forall \vec h\in H (MR(\theta(x,\vec d',\vec h))<n \to \neg \theta(c,\vec d',\vec h))$.

Since $(N,H)\models Th(M,H)$ is $|T|^+$-saturated, we can find $a'$ such that 
$MR(a'/\vec b)\geq n$, $\models \varphi(a',b)$ and whenever $\vec h\in H(N)$ and $\theta(x,\vec b,\vec h)$ is a formula 
with Morley rank smaller than $n$ we have $\neg \theta(a',\vec b,\vec h)$. This shows that 
$MR(a'/\vec bH)=MR(a'/\vec b)=MR(a/\vec b)$, $Md(a'/\vec b)=Md(a/\vec b)$, both $a$ and $a'$ are generics of the formula 
$\varphi(x,\vec b)$ and thus $\tp(a/\vec b)=\tp(a'/\vec b)$. Finally by construction $a'\ind_{\vec b}H$. 
It follows that $T^{ind}_\mathcal{G}$ is first order.

\subsection{$H$-triples}
Recall from \cite{BeVa} that if $T_0$ is supersimple $SU$-rank one theory,  then $T=T_0^{ind}$ is supersimple (of rank $1$ or $\omega$, depending on whether the geometry $T_0$ is trivial or not). The models of $T$ are structures
of the form $(M,H_1)$, where $M\models T_0$ and $H_1$ is a $\acl_0$-dense and $\acl_0$-codense subset of $M$.
We write $\cL_0$ for the language associated to $T_0$ and $\cL$ for the language associated
to $T$. Similarly, we write $\acl_0$ for the algebraic closure in the language $\cL_0$ and for
$A\subset M\models T_0$, we write $S_n^0(A)$ for the space of $\cL_0$-$n$-types over $A$.

In this subsection we  change our notation  and we let $H_2$ be a new predicate symbol that will be interpreted by a dense and codense independent subset of $(M,H_1)$ (in the sense of forking independence in $T$). 

The structures $(M,H_1,H_2)$ were already studied in \cite{BeVa}. We recall the definitions and the main result.
The main tool for studying $T^{ind}_{\mathcal{G}}$ is to take into account the base theory $T_0$ and use triples.

\begin{defn}\label{defnHtriple}
We say that $(M,H_1(M),H_2(M))$ is an $H$-\emph{triple} associated to $T_0$ if:

\begin{enumerate}
\item $M\models T_0$, $H_1(M)$ is an $acl_0$-independent subset of $M$,
$H_2(M)$ is an $acl_0$-independent subset of $M$ over $H_1$.
\item (Density property for $H_1$) If $A\subset M$ is finite and $q\in S_1^0(A)$ is
non-algebraic, there is  $a\in H_1(M)$ such that $a\models q$.
\item (Density property for $H_2/H_1$) If $A\subset M$ is
finite and $q\in S_1^0(A)$ is non-algebraic, there is
$a\in H_2(M)$ such that $a\models q$ and $a\not \in \acl_0(A\cup
H_1(M))$.
\item (Extension property)  If $A\subset M$ is finite and $q\in S_1^0(A)$ is non-algebraic, there is  $a\in
M$, $a\models q$ and $a\not \in \acl_0(A\cup H_1(M)\cup H_2(M))$.
\end{enumerate}
\end{defn}

It is observed in \cite{BeVa}  that if $(M,H_1(M),H_2(M))$, $(N,H_1(N),H_2(N))$ are
$H$-triples, then $Th(M,H_1(M),H_2(M))=Th(N,H_1(N),H_2(N))$ and we
denote the common theory by $T_0^{tri}$.

We will now apply our construction to $T=T_0^{ind}$. Let $\mathcal{G}$ be the partial type in $T$ representing $M\backslash \acl(H_1(M))$. Then an $H$-structure with respect to $\mathcal{G}$ will be given by $(M, H_1, H_2)$ where $H_2(M)\subset M\backslash\acl(H_1(M))$, $H_2(M)$ is forking-independent in the sense of $(M,H_1)$ (which is equivalent to $\acl(-\cup H_1(M))$-independence or, in the notation of \cite{BeVa}, $\scl$-independence), such that for any finite $A\subset M$ and $a\in M$ such that $a\not\in\acl(AH_1(M))$, $\tp(a/A)$ (in the sense of $(M,H_1)$) is realized in $H_2(M)$ and in $M\backslash \acl(A H_1(M) H_2(M))$. This agrees with the notion of $\scl$-structure in \cite{BeVa}.

The following result is proved (in the more general geometric setting) in \cite{BeVa} (Prop. 4.6).

\begin{prop}
Let $T_0$ be an $SU$ rank 1 supersimple theory, let $M\models T_0$ and let
$H_1(M)\subset M$, $H_2(M)\subset M$ be distinguished subsets. Then
$(M,H_1(M),H_2(M))$ is a $\scl$-structure associated to $T_0$ if and only if $(M,H_1(M),H_2(M))$ is an $H$-triple of the theory $T_0$.
\end{prop}

Thus, to show that the class of $H_2$-structures associated to $T$ is first order, it suffices to prove that this is the case for $H$-triples associated to $T_0$. As pointed out in \cite{BeVa} we have:

\begin{prop}\label{axiomstri}
The theory $T^{tri}$ is axiomatized by:
\begin{enumerate}

\item $M\models T_0$, $H_1(M)$ is an $acl_0$-independent subset of $M$,
$H_2(M)$ is an $acl_0$-independent subset of $M$ over $H_1$.
\item For all $\cL_0$-formulas $\varphi(x,\vec y)$\\
$\forall \vec y (\varphi(x,\vec y)$ nonalgebraic $\implies \exists
x (\varphi(x,\vec y)\wedge x\in H_1)$. 
\item For all
$\cL_0$-formulas $\varphi(x,\vec y)$, $m\in \omega$, and all
$\cL_0$-formulas $\psi(x,z_1,\dots,z_m, \vec y )$ such that for some
$n\in \omega$ $\forall \vec z \forall \vec y \exists ^{\leq n}x
\psi(x,\vec z,\vec y)$
(so $\psi(x,\vec y,\vec z)$ is always algebraic in $x$)\\
$\forall \vec y (\varphi(x,\vec y)$ nonalgebraic $\implies \exists
x (\varphi(x,\vec y)\wedge x \in H_2) \wedge
\\\forall w_1\dots
\forall w_m\in H_1 \neg \psi(x,w_1,\dots,w_m,\vec y))$ \item For
all $\cL_0$-formulas $\varphi(x,\vec y)$, $m\in \omega$, and all
$\cL_0$-formulas $\psi(x,z_1,\dots,z_m, \vec y )$ such that for some
$n\in \omega$ $\forall \vec z \forall \vec y \exists ^{\leq n}x
\psi(x,\vec z,\vec y)$
(so $\psi(x,\vec y,\vec z)$ is always algebraic in $x$)\\
$\forall \vec y (\varphi(x,\vec y)$ nonalgebraic $\implies \exists
x \varphi(x,\vec y) \wedge
\\\forall w_1\dots
\forall w_m\in H_1\cup H_2 \neg \psi(x,w_1,\dots,w_m,\vec y))$

Furthermore, if $(M,H_1,H_2)\models T^{tri}$ is
$|T|^+$-saturated, then $(M,H_1,H_2)$ is an $H$-triple.
\end{enumerate}
\end{prop}

Thus when $T_0$ is a supersimple $SU$-rank one theory,
$T^{ind}_{\mathcal{G}}=T^{tri}$ is first order.

\subsection{$H$ structures of lovely pairs of $SU$-rank one theories}

Let $T$ be a supersimple theory of SU-rank 1, $T_P$ its lovely pairs expansion,
and let $$\ccl(-)=\acl(-\cup P(M))$$ be the small closure operator
in a lovely pair $(M,P)$. Let $\mathcal{G}$ be the formula $\neg P(x)$ in $T_P$  (representing $M\backslash P(M)$).

Our goal is to expand $T_P$ to a theory
$T_P^{ind}$ of $H$-structures of $T_P$ with respect to $\mathcal{G}$,  in the language $\cL_{PH}=\cL_P\cup\{H\}$. Note that the notion of being large in the sense of $\mathcal{G}$ coincides with the corresponding notion induced by the small closure operator $\ccl$.

The following definition is analogous to Definition \ref{defnHtriple}.

\begin{defn}\label{Phstr}
We say that an $\cL_{PH}$-structure $(M,P,H)$ is a $PH$-structure
of $T$ if

\begin{enumerate}

\item $P(M)$ is an elementary substructure of $M$;

\item $H(M)$ is $\acl$-independent over $P(M)$;

\item   for any non-algebraic type $q\in S_1^T(A)$ over a
finite set $A\subset M$, $q$ is realized in

(density of $P$ over $H$) $P(M)\backslash \acl(H(M)A)$;

(density of $H$ over $P$) $H(M)\backslash \acl(P(M)A)$;

(extension) $M\backslash \acl(P(M)H(M)A)$.

\end{enumerate}

\end{defn}

\begin{rema}\label{PH-remark}
(a) It suffices to require $P(M)$ to be dense in the usual sense,
i.e. $q$ having a realization in $P(M)$.Indeed, $H(M)$ is $\acl$-independent over $P(M)$, and if $A$ is finite, then
for some finite $\vec h\subset H(M)$ we have $H(M)\ind_{\vec h P(M)}A P(M)$. Hence $H(M)\backslash\vec h$ is $\acl$-independent over $A P(M)$. Clearly, if $P(M)$ is dense, $q$ can be realized by $b\in P(M)\backslash \acl(\vec h A)$. Then $b\not\in \acl(A H(M))$.

(b) We can get a $PH$-structure from an $H$-triple $(M,H_1,H_2)$ (see previous example), by letting
$P(M)=\acl(H_1)$.

(c) A usual elementary chain argument shows that any $L_{PH}$
structure $(M,P,H)$ satisfying (1,2) embeds in a $PH$-structure
$(N,P,H)$ so that $H(N)\ind_{H(M)} M P(N)$ and $P(N)\ind_{P(M)}
MH(N)$. In particular, $PH$-structures exist.

(d) Reducts $(M,P)$ and $(M,H)$ of $(M,P,H)$ are lovely pairs and
$H$-structures, respectively.

\end{rema}

\begin{prop}
$(M,P,H)$ is an $H$-structure of $T_P$ (with respect to $\mathcal{G}$) if and only if $(M,P,H)$ is a
$PH$-structure.
\end{prop}

\begin{proof}
Assume first that  $(M,P,H)$ is an $H$-structure. Then the pair
$(M,P)$ is lovely and thus $(M,P,H)$ satisfies the density property
for $P$. Now let $A\subset M$ be finite and let $q\in
S_1(A)$ be non-algebraic. Let $\hat q\in S_1^{P}(A)$ be an
extension of $q$ that contains no small formula with parameters in
$A$. Then by the density/coheir property for $\ccl$ it follows
that there is $a\in H(M)$ such that $a\models \hat q$. In
particular, $a\models q$ and $a\not \in \ccl(A)$ and thus we get
the density property for $H$ over $P$. Finally, by extension property, there exists $c\in M$ realizing $\hat q$ such that
$c\ind^P_{A}H(M)$ (here independence is in the sense of $T_P$). Then $c\not\in \ccl(A\cup H(M))=\acl(A\cup P(M)\cup H(M))$. Thus the
extension property of $PH$-structures holds as well.

Now assume that $(M,P,H)$ is a $PH$-structure. Then by the density property for $P$ and
the extension property it follows that $(M,P)$ is a lovely pair, and $H(M)$ is a
$\ccl$-independent set (hence, forking independent in the sense of $T_P$). 
 Now let $A\subset M$ be finite  and let
$\hat q\in S_1^{P}(A)$ be large. We may enlarge $A$ and
assume that $A$ is $P$-independent. Let $q$ be the restriction of
$\hat q$ to the language $\cL$. Note that $\hat q$ is the unique
extension of $q$ to a non-small type. By the density for $H$ over
$P$, there is $a\in H(M)$ such that $a\models q$, $a\not \in
\ccl(A)$ and thus $a\models \hat q$. This proves the density/coheir property of $H$-structures. 

To prove the extension property of $H$-structures, we start with $a\in M$ and a finite $A\subset M$. We may assume that $A$ is $P$-independent.If $a$ is large over $A$ in $T_P$, then by the extension property for $PH$-structures, there is $a'\in M$ such that $\tp(a'/A)=\tp(a/A)$ and $a'\not\in\acl(P(M)H(M))$. Then $\tp^P(a'/A)=\tp^P(a/A)$ (where $\tp^P$ denotes the type in the sense of $T_P$) and $a'\ind^P_A H(M)$, as needed.

Now, suppose $a$ is small over $A$ in $T_P$. Thus, $a\in\acl(AP(M))$. Then it suffices to show that for any  tuple $\vec b\in P(M)$, $\acl$-independent over $A$, there exists $\vec b'\in P(M)$ such that $\tp(\vec b'/A)=\tp(\vec b/A)$ and $\vec b'\ind^P_A H(M)$. We may also assume that $\vec b = b$ is a single element. Suppose for any $b'\in P(M)$ realizing $\tp(b/A)$ we have that $\tp^P(\vec b'/AH(M))$ forks over $A$. Since $P(M)$ has SU-rank 1, this means that $b'\in\acl^P(A H(M))$ for any $b'\in P(M)$ realizing $\tp(b/A)$. Next, as in Remark \ref{PH-remark}(a), we have $H(M)\ind_{\vec h P(M)}A P(M)$. Since $H(M)$ is $\acl$-independent over $P(M)$, we also have $H(M)\ind_{\vec h}A P(M)$. Let $\vec c\in P(M)$ be a finite tuple such that $A \vec h \ind{\vec c} P(M)$. Then $AH(M)\ind_{\vec c} P(M)$. Hence, $AH(M)\vec c$ is $P$-independent, and $\acl^P(AH(M)\vec c)=\acl(AH(M)\vec c)$.
Therefore, $b'\in\acl(A H(M)\vec c)$ for any $b'\in P(M)$ realizing $\tp(b/A)$.  But by density of $P$ over $H$, we can find $b'\in P(M)$ realizing a nonalgebraic extension of $\tp(b/A)$ to $A\vec c$ such that $b'\not\in \acl(A\vec c H(M))$, a contradiction.

\end{proof}

We will now show that the class of $PH$-structures is "first
order", that is, that there is a set of axioms whose
$|T|^+$-saturated models are the $PH$-structures. The axiomatization works
as in $H$-triples.

\begin{prop}\label{axiomsPH}
Assume $T$ eliminates $\exists ^{\infty}$. Then the theory
$T_{PH}$ is axiomatized by:
\begin{enumerate}
\item $T$ \item axioms saying that $P$ distinguishes an elementary
substructure.

\item For all $\cL$-formulas $\varphi(x,\vec y)$\\
$\forall \vec y (\varphi(x,\vec y)$ nonalgebraic $\implies \exists
x (\varphi(x,\vec y)\wedge x\in P))$.

 \item For all
$\cL$-formulas $\varphi(x,\vec y)$, $m\in \omega$, and all
$\cL$-formulas $\psi(x,z_1,\dots,z_m, \vec y )$ such that for some
$n\in \omega$ $\forall \vec z \forall \vec y \exists ^{\leq n}x
\psi(x,\vec z,\vec y)$
(so $\psi(x,\vec y,\vec z)$ is always algebraic in $x$)\\
$\forall \vec y (\varphi(x,\vec y)$ nonalgebraic $\implies \exists
x (\varphi(x,\vec y)\wedge x \in H) \wedge
\\\forall w_1\dots
\forall w_m\in P \neg \psi(x,w_1,\dots,w_m,\vec y))$

 \item For all
$\cL$-formulas $\varphi(x,\vec y)$, $m\in \omega$, and all
$\cL$-formulas $\psi(x,z_1,\dots,z_m, \vec y )$ such that for some
$n\in \omega$ $\forall \vec z \forall \vec y \exists ^{\leq n}x
\psi(x,\vec z,\vec y)$
(so $\psi(x,\vec y,\vec z)$ is always algebraic in $x$)\\
$\forall \vec y (\varphi(x,\vec y)$ nonalgebraic $\implies \exists
x (\varphi(x,\vec y)\wedge x\not \in P\wedge x \not \in H) \wedge
\\\forall w_1\dots
\forall w_m\in P\cup H \neg \psi(x,w_1,\dots,w_m,\vec y))$

Furthermore, if $(M,P,H)\models T_{PH}$ is $|T|^+$-saturated, then
$(M,P,H)$ is a $PH$-structure.
\end{enumerate}
\end{prop}

Now we list a family of structures of $SU$-rank $\omega$ where we do not know 
if the corresponding theory of $H$-structures is axiomatizable. In all three cases 
it is open whether or not the extension property is first order.

\subsection{Vector space with a generic automorphism}

Let $T$ be the theory of an infinite vector space over
a division ring $F$; it is strongly minimal, has DMP and quantifier elimination. 
Recall that the definable closure of a set inside models of $T$ corresponds to the linear 
span of the set. For a tuple $\vec a$ in a model of $T$ we write $qftp^{-}(\vec a)$ for its quantifier free type (which isolates the type) and $\dcl^{-}(\vec a)$ for its definable closure. Let $\sigma$ be a new function symbol 
and let $T_\sigma$ be theory $T$ together with axioms stating that $\sigma$ is a generic 
automorphism. The notes \cite{Pi2} deal with strongly minimal theories with a generic automorphism,
the theory  $T_\sigma$ appears there as an example. In particular it is shown in 
\cite[Proposition 3.9]{Pi2} how to axiomatize $T_\sigma$. By \cite[Corollary 4.14 and Lemma 4.16]{Pi2} 
the theory has $SU$-rank $\omega$ and by \cite[Example 5.7]{Pi2}  it is stable. It is also proved 
in \cite[Remark 3.15]{Pi2} that for $(M,\sigma)\models T_{\sigma}$ and $\vec a\in M$,
the algebraic closure of $\vec a$ in $T_{\sigma}$ denoted by $\dcl(\vec a)$ corresponds to $\dcl^-(\{\sigma^{i}(\vec a): i\in \mathbb{Z}\})$

We claim that $T_{\sigma}$ has quantifier elimination. Indeed, let $(M,\sigma)\models T_{\sigma}$ be saturated
and let  $\vec a,\vec a'\in (M,\sigma)$
be such that $qftp^{-}(\ldots,\sigma^{-1}(\vec a), \vec a
,\sigma(\vec a),\sigma^2(\vec a),\ldots)=$  \\$qftp^{-}(\ldots,\sigma^{-1}(\vec a'), \vec a'
,\sigma(\vec a'),\sigma^2(\vec a'),\ldots)$. Let $c\in M$, we will find $c'\in M$ such that \\
$qftp^-(\ldots,\sigma^{-1}(\vec a c), \vec a c
,\sigma(\vec a c),\sigma^2(\vec a c),\ldots)= qftp^-(\ldots,\sigma^{-1}(\vec a' c'), \vec a'c'
,\sigma(\vec a' c'),\sigma^2(\vec a' c'),\ldots)$. Consider the substructure $\dcl(\vec a, c)$ and the 
free amalgamation of
$(M,\sigma)$ with $\dcl(\vec a, c)$ that identifies the substructure $\dcl(\vec a')$ of $M$ 
with the substructure $\dcl(\vec a)$ of $\dcl(\vec ac)$. Then this amalgamation is a
superstructure of $(M,\sigma)$ and the copy of $c$ in the amalgamation has the desired type. Now use
that $(M,\sigma)$ is existentially closed and saturated to get $c'$.

The theory $T_\sigma$ has a unique type of $U$-rank $\omega$, namely the type of a transformally independent element.
Take $\mathcal{G}$ to be this type, which is definable over $\emptyset$. By definability of types,
$T_{\sigma}$ eliminates the quantifier $\exists^{large}$.

\textbf{Question} Is the extension property first order for $T_\sigma$?

We can choose other partial types for $\mathcal{G}$. For example, 
let $\mathcal{G}$ be $\sigma(x)=x$, $x\not \in \dcl^-(\emptyset)$. By quantifier elimination
and stability the set $\sigma(x)=x$ is strongly minimal. The partial type 
$\mathcal{G}$ corresponds to the generic type of this strongly minimal set, which is stationary 
and definable  over $\emptyset$. The definability of the type implies elimination 
of $\exists^{large}$.

\subsection{ACFA}
Let $T=ACFA$, (a completion) of the theory of algebraically closed fields
with a generic automorphism. This theory is simple of $SU$-rank $\omega$
and it is unstable.

Let $p(x)$ be the generic type of the theory,
namely the type of a transformally independent element. This type is complete,
stationary and definable over $\emptyset$. Let $\mathcal{G}=p(x)$. Let $\varphi(x,\vec y)$ be a formula
and let $\psi(\vec y)$ be its $p$-definition. Then for
$(K,\sigma)\models ACFA$, $\vec a\in K$, the formula $\varphi(x,\vec a)$ is
large iff $\psi(\vec a)$. Thus this theory eliminates the quantifier
$\exists^{large}$.

\textbf{Question} Is the extension property first order for ACFA? Does ACFA satisfy wnfcp?

\subsection{Hrushovski amalgamation without collapsing}

In this subsection we follow the presentation of Hrushovski amalgamations from \cite{Zi}, all the results we mention
can be found in \cite{Zi}.
Let $\cL=\{R\}$ where $R$ stands for a ternary relation. We let $\mathcal{C}$ be the class of $\cL$-structures
where $R$ is symmetric and not reflexive. For $A\in \mathcal{C}$ a finite structure we let 
$\delta(A)=|A|-|R(A)|$ and we let  $\mathcal{C}^0_{fin}$ be the subclass of $\mathcal{C}$ consisting of all finite $\cL$-structures $M$ where for $A\subset M$ we have $\delta(A)\geq 0$. Finally $M^0$ stands for the Fra\"iss\'e limit of the 
class $\mathcal{C}^0_{fin}$. Let $T_0$ be the theory of $M^0$, then $MR(T_0)=\omega$ and $Md(T_0)=1$. 

Now let $M\models T_0$ and for $A\subset M$ finite we define $d(A)=\inf\{\delta(B): A\subset B\}$. Then $d$ is the dimension function of a pregeometry and that for an element $a$ and a set $B$, $d(a/B)=1$ if and only if $MR(a/B)=\omega$ if and only if $U(a/B)=\omega$. The type $p(x)$ described by $d(x)=1$ is stationary and definable over $\emptyset$, if we choose $\mathcal{G}=p(x)$ then $T_0$ eliminates the quantifier $\exists^{large}$.

\textbf{Question} Is the extension property first order for $T_0$? Does $T_0$ satisfy nfcp?

\section{Definable sets in $H$-structures}\label{sec:definablesets}

Fix $T$ a supersimple theory, $\mathcal{G}$ a type-definable set over $\emptyset$, 
and assume that $T$ eliminates $\exists^{large}$ and that the extension property is first order.
Let $(M,H(M))\models T^{ind}_{\mathcal{G}}$. Our next
goal is to obtain a description of definable subsets of $M$ and
$H(M)$ in the language $\cL_H$.

\begin{notation}
Let $(M,H(M))$ be an $H$-structure. Let $\vec a$ be a tuple in $M$.
We denote by $\etp_H(\vec a)$ the collection of formulas of the
form $\exists x_1\in H\dots \exists x_m\in H\varphi(\vec x,\vec y)$, where
$\varphi(\vec x,\vec y)$ is an $\cL$-formula such that
there exists $\vec h\in H$ with $M\models \varphi(\vec h,\vec a)$.
\end{notation}

The next results appears in \cite{BV-Tind} in the setting of geometric theories. The proof
we present here is a translation of Corollary 3.11 \cite{BPV} to the setting of
$H$-structures.

\begin{lem}\label{AlmostQE}
Let $(M,H(M))$, $(N,H(N))$ be $H$-structures. Let $\vec a$, $\vec b$
be tuples of the same arity from $M$, $N$ respectively. Then the
following are equivalent:
\begin{enumerate}
\item $\etp_H(\vec a)=\etp_H(\vec b)$.
\item $\vec a$, $\vec b$ have the same $\cL_H$-type.
\end{enumerate}
\end{lem}

\begin{proof}
Clearly (2) implies (1). Assume (1), then  $\tp(\vec a)=\tp(\vec b)$.

Let $\vec h_{\vec a}$ be an enumeration of a finite subset of $H(M)$ such that $\tp(\vec a/H(M))$ does not fork over $\vec h_{\vec a}$. Let $q(\vec z, \vec a)=\tp(\vec h_{\vec a},\vec a)$. Since $\etp_H(\vec a)=\etp_H(\vec b)$, $q(z,\vec b)$ is finitely satisfiable in $H(N)$, and since $(N,H(N))$ is an $H$-structure, it is realized in $H(N)$ say by $\vec h_{\vec b}$.

\textbf{Claim} $tp(\vec b/H(N))$ does not fork over $\vec h_{\vec b}$. 

Suppose, in order to get a contradiction, that $tp(\vec b/H(N))$ forks over $\vec h_{\vec b}$. Then there are $\cL$-formulas $\varphi(\vec x,\vec y)$, $\psi(\vec x,\vec z)$, $k < \omega$,  and 
$\vec c \in H(N)$ such that $\psi(\vec b,\vec c)$ holds and $D(\psi(\vec x,\vec c),\varphi,k) < n$ where $n = D(\tp(\vec b/\vec h_{\vec b}),\varphi,k)$. Now there is a formula $\chi (\vec z) \in \tp(\vec c)$ such that for any $\vec c'\models \chi(\vec z)$, 
$D(\psi(\vec x, \vec c'), \varphi, k) < n$. As $\etp_H(\vec a)=\etp_H(\vec b)$, there is $\vec c' \in H(M)$ such that $\psi(\vec a,\vec c')\wedge \chi(\vec c')$, implying that $tp(\vec a/H(M))$ forks over $\vec h_{\vec a}$, a contradiction.
So $tp(\vec b/H(N))$ does not fork over $\vec h_{\vec b}$. Note that $\tp(\vec a \vec h_{\vec a})=\tp(\vec b\vec h_{\vec b})$ and both tuples
are $H$-independent, so by Proposition \ref{H-indtypes} $\tp_H(\vec a)=\tp_H(\vec b)$.

\end{proof}

Note that Lemma \ref{AlmostQE} hold for tuples in $M^{eq}$. 

The previous result implies near model completeness:  every $\cL_H$-definable set can be written as a boolean combination
of formulas of the form $\exists x_1\in H\dots \exists x_m\in H\varphi(\vec x,\vec y)$, where
$\varphi(\vec x,\vec y)$ is an $\cL$-formula. 
We will refine this result, starting by understanding the 
$\cL_H$-definable subsets of $H(M)$.
This material is very similar to the results presented in \cite{BV-Tind}.

\begin{lem}\label{Hsubsets}
Let $(M_0,H(M_0))\preceq (M_1,H(M_1))$ and assume that
$(M_1,H(M_1))$ is $|M_0|$-saturated. Then $M_0$ (seen as a subset
of $M_1$) is an $H$-independent set.
\end{lem}

\begin{proof}
Assume not. Then there is $\vec a\in M_0$
such that $\vec a \nind_{H(M_0)} H(M_1)$. Let $\vec h_0\in H(M_0)$ be such that
$\vec a\ind_{\vec h_0} H(M_0)$, then $\vec a \nind_{\vec h_0} H(M_1)$ and let $H_1\subset H(M_1)\setminus H(M_0)$ be finite and $h \in H(M_1)\setminus (H(M_0)\cup H_1)$ such that $\vec a \nind_{\vec h_0} H_1h$, but  $\vec a \ind_{\vec h_0} H_1$. So $h$ forks with $\vec a \vec h_0 H_1$. 
Let $\varphi(x,\vec a;\vec h_0, H_1)$ be a formula which holds for $h$ such
that $\varphi(x,\vec a;\vec h_0, H_1) \wedge \mathcal G(x)$ forks over $\emptyset$. Since
$(M_0,H(M_0))\preceq (M_1,H(M_1))$ there is $h_2\in H(M_0)$ and $H_2\subset
H(M_0)$ both disjoint from each other and disjoint from $\vec h_0$ such that 
$\varphi(h_2,\vec a,\vec h_0,H_2) \wedge \neg \exists^{large}x 
\varphi(x,\vec a,\vec h_0,H_2)$ holds, so $\vec a\nind_{\vec h_0}h_2H_2$, a
contradiction.
\end{proof}

\begin{prop}\label{induced}
Let $(M,H(M))$ be an $H$-structure and let $Y\subset H(M)^n$ be
$\cL_H$-definable. Then there is $X\subset M^n$ $\cL$-definable
such that $Y=X\cap H(M)^n$.
\end{prop}

\begin{proof}
Let $(M_1,H(M_1))\succeq (M,H(M))$ be $\kappa$-saturated where
$\kappa>|M|+|L|$ and let $\vec a,\vec b\in H(M_1)^n$ be such that
$\tp(\vec a/M)=\tp(\vec b/M)$. We will prove that $\tp_H(\vec
a/M)=\tp_H(\vec b/M)$ and the result will follow by compactness.
Since $\vec a, \vec b \in H(M_1)^n$, we get by Lemma
\ref{Hsubsets} that $M\vec a$, $M\vec b$ are $H$-independent sets
and thus by Lemma \ref{H-indtypes} we get $\tp_H(\vec a/M)=\tp_H(\vec
b/M)$.
\end{proof}

%The collection $\mathcal{G}$ is a family of 
%regular types such that their forking extensions are orthogonal not only 
%to the original type but to the full family of types in $\mathcal{G}$. For example, 
%when $T$ has SU-rank one, we can take $\mathcal{G}$ to  be the collection of 
%types of SU-rank one which is type-definable. In this setting, if an element forks, it becomes algebraic and thus
%ortogonal to $\mathcal{G}$. Similarly, consider $T$ a theory of $SU$-rank $omega$ and choose $\mathcal{G}$ as the family of all %types of $SU$-rank $\omega$ and assume that $\mathcal{G}$ is type-definable. In this setting, if an element $a$ forks
%over a set $B$, $SU(a/B)<\omega$ and by Lascar's inequality it is ortogonal to $\mathcal{G}$.

%\begin{defn}
%Let $(M,H)\models T^{ind}$ be $\kappa$-saturated and let $A\subset M$ be smaller than
%$\kappa$. Let $\vec b\in M$ be a tuple. 
%\end{defn}

\begin{prop}\label{H-basis}
Let $(M,H(M))$ be an $H$-structure. Let $\vec a \in
M$. Then there is a unique smallest subset $H_0\subset H(M)$ such
that $\vec a\ind_{H_0} H$.
\end{prop}

\begin{proof}
Since $T$ is supersimple, there is a finite subset $H_1\subset H(M)$ such that $\vec a\ind_{H_1} H(M)$. Choose such
 subset so that $|H_1|$ (the size of the subset) is minimal. We will now show such a set $H_1$ is unique.

If $\vec a\ind H(M)$, then $H_1=\emptyset$ and the result
 follows. So we may assume that $\vec a\nind  H(M)$.

 Let $H_2\subset H(M)$ be another minimal finite subset such that $\vec a\ind_{H_2} H(M)$ and let $H_0=H_1\cap H_2$.

 \textbf{Claim} $\vec a\ind_{H_0}H(M)$.

Since the elements in $H(M)$ are independent over $\emptyset$, we have $H_1\ind_{H_0}H_2$. Since $\vec a\ind_{H_1} H(M)$
we also have $\vec a\ind_{H_1} H_2$ and by transitivity $\vec a\ind_{H_0} H_2$. But $\vec a\ind_{H_2} H(M)$ , so using again
transitivity we get $\vec a\ind_{H_0} H(M)$. 

Finally by minimality of $|H_1|$ we get that $H_0=H_1=H_2$ as we wanted. 
 \end{proof}

 \begin{rema}\label{H-basisrel}
Let $(M,H(M))$ be an $H$-structure. Let $\vec a\in M$ and
let $C\subset M$ be $H$-independent. Then whenever
$H_1,H_2$ are finite subsets of $H$ and $H_0=H_1\cap H_2$, 
we have $H_1 \ind_{H_0C} H_2$. Then the argument from the previous proposition
relativized to $C$ shows that there is a unique smallest subset
$H_0\subset H(M)$ such that $\vec a\ind_{H_0C} H$.
\end{rema}

 \begin{notation}
 Let $(M,H(M))$ be an $H$-structure. Let $\vec a\in M$.
Let $H_0\subset  H(M)$ be the smallest subset such that $\vec a\ind_{H_0} H$. We call $H_0$ the 
\emph{$H$-basis} of $\vec a$ and we denote it as $HB(\vec a)$. Given $C\subset M$ such that
$C$ is $H$-independent, let $H_1\subset H(M)$ be the
smallest subset such that $\vec a\ind_{C H_1} H$. We call $H_1$ the \emph{$H$-basis of $\vec a$ over $C$} and we denote it as
$HB(\vec a/C)$. Note that $H$-basis is a finite set, if we give it an order to view it as a
tuple, we will explicitly say so. Finally note that we can also define the $H$-basis for $\vec a\in M^{eq}$.
 \end{notation}

\begin{prop}
Let $(M,H(M))$ be an $H$-structure. Let $a_1,\dots,a_n,a_{n+1}\in M$ and
let $C\subset M$ be such that $C$ is $H$-independent. Then
$HB(a_1,\dots,a_n,a_{n+1}/C)=HB(a_1,\dots,a_n/C)\cup HB(a_{n+1}/Ca_1,\dots,a_nHB(a_1,\dots,a_n/C))$.
\end{prop}

\begin{proof}
Let $H_1=HB(a_1,\dots,a_n/C)$.
First note that since $a_1,\dots,a_n \ind_{C H_1}H$, then the set $a_1,\dots,a_n C H_1$ is
$H$-independent and we can define $H_2= HB(a_{n+1}/Ca_1,\dots,a_n H_1)$.
Finally, let $H_0=HB(a_1,\dots,a_n,a_{n+1}/C)$.

\textbf{Claim} $H_0\subset H_1H_2$.

We have $a_1,\dots,a_n \ind_{C H_1}H$ and $a_{n+1} \ind_{C H_1 H_2a_1,\dots,a_n}H$, so
by transitivity,

 $a_1,\dots,a_n a_{n+1}\ind_{C H_1 H_2}H$ and by the minimality of an $H$-basis,
we have $H_0\subset H_1 H_2$.

\textbf{Claim} $H_0 \supset H_1 H_2$.

By definition, $a_1,\dots,a_n a_{n+1}\ind_{C H_0}H$, so $a_1,\dots,a_n \ind_{C H_0}H$ and by minimality
we have $H_1\subset H_0$. We also get by transitivity that $a_{n+1}\ind_{Ca_1,\dots,a_n H_1 H_0}H$
and by the minimality of $H$-basis we get $H_2\subset H_0$ as desired.
\end{proof}

\begin{prop}\label{changinghb}
Let $(M,H(M))$ be an $H$-structure. Let $\vec a \in M$ and
let $C\subset D\subset M$ be such that $C$, $D$ are $H$-independent. Then $HB(\vec a/C)\subset HB(D)\cup HB(\vec a/D)$. 
\end{prop}

\begin{proof}
Write $H_{\vec a}=HB(\vec a/D)$. Then $\vec a D\ind_{H_{\vec a} H(D)}H$ and $\vec a\ind_{H_{\vec a} C H(D)}H$. By minimality of $HB(\vec a/C)$ we get that $HB(\vec a/C)\subset H_{\vec a} \cup H(D)$ and thus if $h\in HB(\vec a/C)\setminus H_{\vec a}$, we must have $h\in H(D)$.
\end{proof}

The following definition and proposition were very fruitfull to show the preservation of NTP2 to $T^{ind}$ 
when $T$ was geometric see \cite{BeKi}.

\begin{defn}
Let $(M,H)\models T^{ind}_{\mathcal{G}}$ be saturated. We say that an $\cL_H$-formula $\psi(x,\vec c)$ defines
a \emph{H-large} subset of $M$ if there is 
$b\models \psi(x,\vec c)$ such that $b\ind H(M)\vec c$, $b \models \mathcal{G}$. Otherwise we say that
$\psi(x,\vec c)$ defines a \emph{H-small} subset of $M$.
Note that the formula $\psi(x,\vec c)$ is $H$-large if there are infinitely many realizations of $\psi(x,\vec c)$ in $\mathcal{G}$ that
are independent from $H(M)\vec c$. For $A\subset M$ and $b\in M$, we say that $b$ is 
\emph{H-small} over $A$ if it satisfies an $H$-small formula $\psi(x,\vec a)$ with $\vec a\in A$, otherwise we say that
$b$ is \emph{H-large} over $A$.
%of $A$ if $\vec b\in \ccl(AH(M))$. Let $X\subset M^n$ be $A$-definable. We say that $X$ is
%\emph{H-small} if $X\subset \cl(A\cup H(M))$.
\end{defn}

 \begin{prop}\label{uptosmall}
Let $(M,H(M))$ be an $H$-structure and let $Y\subset M$ be
$\cL_H$-definable. Then there is $X\subset M$ $\cL$-definable
such that $Y\triangle X$ is $H$-small, where $\triangle$ stands for a boolean
connective for the symmetric difference.
\end{prop}

\begin{proof}
If $Y$ is $H$-small or its complement is $H$-small, the result is clear, so we may assume that both $Y$
and $M\setminus Y$ are $H$-large. Assume that $Y$ is definable over $\vec a$
and that $\vec a=\vec a HB(\vec a)$.
Let $b\in Y$ be such that $b\models \mathcal{G}$, $b\ind H(M)\vec a$ and let $c\in M\setminus Y$
be such that $c\models \mathcal{G}$, $c\ind H(M)\vec a$. Then $b\vec a$, $c\vec a$ are
$H$-independent and thus there is $X_{bc}$ an $\cL$-definable set such that
$b\in X_{bc}$ and $c\not \in X_{bc}$. By compactness, we may get a single 
$\cL$-definable set $X$ such that for all $b'\in Y$ and $c'\in M\setminus Y$ $H$-large over $\vec a$, we have $b'\in X$ and $c'\in M\setminus X$. This shows that
$Y\triangle X$ is $H$-small.
\end{proof}

 Our next goal is to characterize the algebraic closure in $H$-structures. The key
 tool is the following result:

\begin{lem}\label{aclp}
Let $(M,H(M))$ be an $H$-structure, and
let $A\subset M$ be $\acl$-closed and $H$-independent.
Then $A$ is $\acl_H$-closed.
\end{lem}

\begin{proof}
Suppose $a\in M$, $a\not\in A$. We do the argument by cases:

Case 1. Assume $a\ind_AH(M)$. Using repeatedly the extension
property, we can find $\{a_i :i\in\omega\}$ all of them realizing $\tp(a/A)$ and 
independent over $A\cup H(M)$. By Proposition \ref{H-indtypes},
each $a_i$ realizes $\tp_H(a/A)$, and thus $a\not\in\acl_H(A)$.

Case 2. Assume $a\nind_AH(M)$. Let $\vec h\in H(M)$ be an enumeration of $HB(a/A)$ (so we see it as a tuple), then
$a\ind_{A\vec h}H(M)$. Using the density property we can find $\{\vec h_i :i\in\omega\}$ in $H$ all of them realizing $\tp(\vec h/A)$ and disjoint one from the other. Let $a_i\in M$ be such that $\tp_H(a_i,\vec h_i/A)=\tp_H(a,\vec h/A)$. Then for each $i$, $HB(a_i/A)=\vec h_i$ (again with an order) and so the family $\{a_i:i\in \omega\}$ is infinite and they all satisfy $tp_H(a/A)$.
\end{proof}

\begin{cor}\label{aclp2}
Let $(M,H(M))$ be an $H$-structure, and
let $A\subset M$. Then $\acl_H(A)=\acl(A,HB(A))$.
\end{cor}

\begin{proof}
By Proposition \ref{H-basis}, it is clear that $HB(A)\subset \acl_H(A)$,
so $\acl_H(A)\supset \acl(A,HB(A))$. On the other hand, $A\cup HB(A)$
is $H$-closed, so by the previous proposition,  $\acl(A\cup HB(A))=
\acl_H(A\cup HB(A))$ and thus $\acl_H(A)\subset \acl(A,HB(A))$
\end{proof}

\section{Supersimplicity}\label{sec:supersimplicity}
Fix $T$ a supersimple theory, $\mathcal{G}$ a type definable set over $\emptyset$, 
and assume that $T$ eliminates $\exists^{large}$ and that the extension property is first order.
 In this section we prove that $T^{ind}_\mathcal{G}$ is supersimple and characterize forking in $T^{ind}_{\mathcal{G}}$.

\begin{thm}\label{simplicity}
The theory $T^{ind}_\mathcal{G}$ is supersimple.
\end{thm}

\begin{proof}
We will prove that non-dividing has local character.

Let $(M,H(M))\models T^{ind}_\mathcal{G}$ be saturated. Let $C\subset D\subset
M$ and assume that $C=\acl_H(C)$ and $D=\acl_H(D)$.
Note that both $C$ and $D$ are $H$-independent. Let $\vec a\in M$.
We will find  a collection of conditions for the type of $\vec a$ over
$C$ that guarantee that $\tp_H(\vec a/D)$ does not divide
over $C$.

Assume that the following conditions hold for $C$:
\begin{enumerate}
\item $HB(\vec a/D)=HB(\vec a/C)$.
\item $\vec a\ind_{CH(M)}DH(M)$
\end{enumerate}

\textbf{Claim} $\tp_H(\vec a/D)$ does not divide over $C$.

Let $\{D_i:i\in \omega\}$ be an
$\cL_H$-indiscernible sequence over $C$. Let $\vec h$ be an enumeration of $HB(\vec a/D)$ and let 
$q(\vec x,D)=\tp(\vec h,D)$. Since $D$ is $H$-independent, the tuple $\vec h$ is independent over $D$
and all components of $\vec h$ realize $\mathcal{G}$.
Thus we can find $\vec g\models \cup _{i\in \omega} q(\vec x,D_i)$ such that
$\{\vec gD_i:i\in \omega\}$ is indiscernible and $\vec g$ is an
independent tuple over $\cup _{i\in \omega} D_i$. All components of $\vec g$ realize $\mathcal{G}$. 
By the generalized
density property, we may assume that $\vec g$  is in $H$. Note that $\vec h D$ is
$H$-independent and $\vec g D_i$ is also $H$-independent for any
$i\in \omega$. So by Proposition \ref{H-indtypes} $\tp_H(\vec h
D)=\tp_H(\vec g D_i)$ for any $i\in \omega$. 

Let $\vec e$ be such that $\tp_H(\vec a,D,\vec h)=\tp_H(\vec e,D_0,\vec g)$.
Let $p(\vec x,D,\vec z)=\tp(\vec a,D,\vec h)$. Then $\vec e\models p(\vec x, D_0, \vec g)$.
Since $\vec a\ind_{CH(M)}DH(M)$, we get that $\vec a\ind_{C\vec h}D\vec h$ and so
$\vec e\ind_{C\vec g}D_0\vec g$. Since $\{\vec gD_i:i\in \omega\}$ is an 
$\cL$-indiscernible sequence over $C\vec g$, we can find $\vec a'\models \cup _{i\in \omega} p(\vec x,D_i,\vec g)$ such that $\{\vec a'\vec gD_i:i\in \omega\}$ is $\cL$-indiscernible and $\vec a'$ is
independent from $\cup _{i\in \omega} D_i\vec g$ over $C\vec g$. By the generalized
extension property, we may assume that $\vec a'\ind_{\cup_{i\in \omega}D_i\vec g}H(M)$ 
and by transitivity we get $\vec a'\ind_{C\vec g}D_iH(M)$ and so $\vec a'\ind_{D_i\vec g}H(M)$
for each $i\in \omega$. Note that $\vec a \vec h D$ is
$H$-independent and $\vec a' \vec g D_i$ is also $H$-independent for any
$i\in \omega$. Thus, by Proposition \ref{H-indtypes}, $\tp_H(\vec a \vec h D)=\tp_H(\vec a' \vec g D_i)$ for 
any $i\in \omega$.

This shows that $\tp_H(\vec a/D)$ does not divide over $C$.

Since $T$ is supersimple, for any $D$ and $\vec a$ we can always choose a finite
subset $C_0$ of $D$ such that $C=\acl_H(C_0)$ satisfies the conditions (1) and (2) above. This shows that $T^{ind}_{\mathcal{G}}$ is supersimple.
\end{proof}

\begin{notation}
Whenever $(M,H(M))\models T^{ind}_\mathcal{G}$ is sufficiently saturated and $A,B,C\subset M$,
we write $A\ind^{ind}_CB$ for $tp_H(A/BC)$ does not fork over $C$. 
\end{notation}

\begin{thm}\label{charforking}
Let $(M,H)\models T^{ind}_\mathcal{G}$ be saturated, let $C\subset D\subset M$ be
such that $C=\acl_H(C)$, $D=\acl_H(D)$ and let $\vec a\in M$.
Then $\tp_H(\vec a/D)$ does not fork over $C$ if and only if:
\begin{enumerate}
\item $HB(\vec a/D)=HB(\vec a/C)$.
\item $\vec a\ind_{CH(M)}DH(M)$
\end{enumerate}
\end{thm}

\begin{proof}
In the proof of Theorem \ref{simplicity} we showed that if the two conditions listed hold, then 
$\tp_H(\vec a/D)$ does not fork over $C$. It remains to show the other direction. 

\textbf{Case 1:} Assume that $\vec a\nind_{CH(M)}DH(M)$. 

%Let $H_0=HB(\vec a/D)$ so $H_0\subset \acl_H(\vec a D)$. 
%We claim that $\vec a\nind_{CH_0H(D)}D$. Otherwise $\vec a \ind_{CH_0H(D)}D$ and since $\vec a %\ind_{DH_0} H(M)$, we get  $\vec a \ind_{CH_0H(D)}DH(M)$ and thus $\vec a \ind_{CH(M)}DH(M)$
%a contradiction.

\textbf{Claim} $\vec a\nind^{ind}_{CH(D)}D$.

Let $p(\vec x, \vec y, CH(D))=\tp_H(\vec a, D / CH(D))$, where we view 
$D$ as an ordered tuple. Let $\{D_i:i\in \omega\}$ be an $\cL$-Morley
sequence in $\tp(D/CH(D))$. By the generalized extension property, we may assume that
$\{D_i:i\in \omega\}$ is independent from $H(M)$ over $CH(D)$. In particular, for every $i$,
$D_i\ind_{CH(D)} H(M)$. Since $C\ind_{H(C)}H(M)$ (as $C=\acl_H(C)$) then $C\ind_{H(D)}H(M)$ and by transitivity 
$D_i\ind_{H(D)} H(M)$, so $H(D_i)\subset H(D)$. Moreover, since $D_i\models \tp(D/CH(D))$ and $H(D)\subset D$, we have that $H(D)\subset D_i$. Thus $H(D)\subset H(D_i)$, i.e. $H(D)=H(D_i)$. 

By Proposition \ref{H-indtypes}, $\tp_H(D)=\tp_H(D_i)$. Furthermore since $\{D_i:i\in \omega\}$ is Morley over $CH(D)$ and independent from $H(M)$ over $CH(D)$, $\{D_i:i\in \omega\}$ is an $\cL_H$-indiscernible 
sequence over $CH(D)$. We will show that $\cup_{i\in \omega}p(x,D_i, CH(D))$ is inconsistent.
Assume, in order to get a contradiction, that there is
$$\vec a'\models \cup_{i\in \omega}p(x,D_i, CH(D))$$ 

%Since $\tp_H(\vec a,DH(D))=\tp_H(\vec a',D_iH(D))$, if we define $H_i'= HB(\vec a'/D_i)$ then $\tp_H(\vec %a,DH(D)H_0)=\tp_H(\vec a',D_iH(D)H_i')$, 
%and since  $\vec a\nind_{CH(D)H_0}D$ we get $\vec a'\nind_{CH(D)H_i'}D_i$. 

Since $\tp_H(\vec a D/CH(D))=\tp_H(\vec a' D_i/CH(D))$ and $\vec a\nind_{CH(M)}DH(M)$ we also
get that $\vec a'\nind_{CH(M)}D_iH(M)$. 

This shows that $\vec a'$ forks with each term in the independent sequence 
$\{D_iH(M):i\in \omega\}$ over $CH(M)$, a contradiction to local character in $T$.

We just showed that $\vec a\nind^{ind}_{CH(D)}D$. On the other hand, if $\vec a\ind^{ind}_{C}D$, since $H(D)\subset D$ we also get $\vec a\ind^{ind}_{CH(D)}D$, a contradiction. Thus we get $\vec a\nind^{ind}_{C}D$ as we wanted.

%This proves that $\tp_H(\vec a/DCH(D)H_0)$ forks over $CH(D)H_0$. Since 
%$H_0\subset \acl_H(\vec a C)$ we also get that $\tp_H(\vec a/DH(D))$ forks over $CH(D)$
%as we wanted.

\textbf{Case 2:} Assume that $\vec a\ind_{CH(M)}DH(M)$ and that $HB(\vec a/D)\neq HB(\vec a/C)$. 

We will first prove:

\textbf{Claim} $HB(\vec a/D)$ is a proper subset of $HB(\vec a/C)$. 

Write $H_D=HB(\vec a/D)$ and  $H_C=HB(\vec a/C)$. Since
$\vec a\ind_{CH_C}H$, from $\vec a\ind_{CH}DH$ we get $\vec a\ind_{CH_C}DH$
and so $\vec a\ind_{DH_C}H$. By minimality of $H$-basis, we have $H_D\subset H_C$
and since the two sets are not equal we get the claim.

Let $H_E=H_C\setminus H_D$, $H_E$ is a set of independent elements over $C$.

Assume, in order to get a contradiction, that $\tp_H(\vec a/D)$ does not fork over $C$.
Then $\vec a\ind^{ind}_CD$. By Proposition \ref{changinghb}, we get that $H_E\subset D$ and by Corollary \ref{aclp2}, $H_E\subset \acl_H(\vec a C)$, so $H_E\ind^{ind}_C H_E$, a contradiction

%Let $p(\vec x,\vec y)=tp_H(\vec a,H_E/C)$. Let $\{H_E^i: i\in \omega\}$ be
%an $\cL$-indiscernible sequence in $\tp(H_E/C)$ such that
%$\{H_E^i: i\in \omega\}$ is independent over $C$. By the generalized
%density property, we may assume that the sequence $\{H_E^i: i\in \omega\}$
%belongs to $H$. Note that by Proposition \ref{H-indtypes}, the sequence
%$\{H_E^i: i\in \omega\}$ is $\cL_H$-indiscernible over $C$. We will show that
%$\cup_{i\in \omega}p(\vec x,H_E^i)$ is inconsistent.
%Assume, not, so there is $\vec a'\models \cup_{i\in \omega}p(x,H_E^i)$. Then
%we can find $H_{D_i}$ in $H$ such that
%$HB(\vec a'/C)=H_{D_i} H_E^i$. Since the $H_E^i$ are independent, we get
%that the $H$-basis of $\vec a'$ over $C$ is not unique, a contradiction.

\end{proof}

%We use the above result to give a different perspective on $H$-basis.

%\begin{lem}\label{H-basisrel}
%Let $(M,H(M))$ be an $H$-structure. Let $\vec a=(a_1,\dots,a_n)\in M$ and
%let $C\subset M$ be such that $C$ is $H$-independent. Let $\vec h$ be a minimal 
%tuple such that $\dim_{\ccl}(\vec a/C\vec h)=\dim_{\ccl}(\vec a/CH)$, then 
%$\vec h=HB(\vec a/C)$.
%\end{lem}

%\begin{proof}
%Write $\vec a=\vec a_1\vec a_2$, where $\vec a_1$ are independent generics over $CH$
%and $\vec a_2\in \ccl(\vec a_1CH)$. Choose $\vec h$ minimal so that 
%$\vec a_2\in \ccl(\vec a_1 C\vec h)$. Then $\vec a_1$ are independent generics 
%over $CH$ and $SU(\vec a_2/\vec a_1C\vec h)<\omega$. Then $\tp(\vec a_1/C)$
%is independent from $H$ and $\tp(\vec a_2/\vec a_1C\vec h)$ is orthogonal to $H$.
%We get $\vec a\ind_{C\vec h}H$ and $HB(\vec a/C)\subset \vec h$. For the other direction, 
%$\vec a\ind_{C HB(\vec a/C)}H$ implies that 
%$\dim_{\ccl}(\vec a/CHB(\vec a/C))=\dim_{\ccl}(\vec a/CH)$ and by minimality of $\vec h$
%we get $HB(\vec a/C)\subset \vec h$.
%\end{proof}

We are interested in characterizing canonical bases. Note that we can work in $T^{eq}$, and thus, $\acl=\acl^{eq}$. We start with the following result which
holds also in the geometric setting (with $e\in\acl^{eq}(B)$):

\begin{lem}\label{precb}
Let $(M,H)\models T^{ind}_\mathcal{G}$ be  sufficiently saturated,
$B\subset M$ be an $H$-independent set, and $\vec a\in M$, $
h=HB(\vec a/B)$. Suppose $e\in\acl^{eq}(B)$ (in the sense of $T$) is such
that $\vec a h\ind_eB$. Then $\vec a\ind^{ind}_eB$.
\end{lem}

\begin{proof}
We use the characterization of forking given in Theorem \ref{charforking}.  

We first prove that $\vec a\ind_{\acl_H(e)H(M)}BH(M)$. Since $h=HB(\vec a/B)$, we have that $\vec a\ind_{Bh}BH(M)$ and thus 
$\vec a\ind_{Beh}BH(M)$. By assumption $\vec a h\ind_eB$, so  $\vec a \ind_{eh}B$ and
by transitivity $\vec a\ind_{eh}BH(M)$ and we get $\vec a\ind_{eH(M)}BH(M)$ and 
$\vec a\ind_{\acl_H(e)H(M)}BH(M)$.

Now we prove that $HB(\vec a/B)=HB(\vec a/\acl_H(e))$.  Let $h_0=HB(\vec a/\acl_H(e))$. We have 
$\vec a\ind_{B h}H(M)$ and since $e\in\acl^{eq}(B)$ and $\acl(B)=\acl_H(B)$ we 
get $\vec a \ind_{B\acl_H(e) h}H(M)$. But $\vec a \ind_{\acl_H(e)h}B$, so by transitivity 
we get  $\vec a \ind_{\acl_H(e)h}H(M)$. This shows that $h\supset h_0$.

 Since $\vec a h\ind_{\acl_H(e) }B$ we get 
$\vec a \ind_{\acl_H(e) h}B$ and by transitivity $\vec a \ind_{\acl_H(e) h}BH(M)$. On the other hand, 
$\vec a \ind_{\acl_H(e) h_0}H(M)$, so $\vec a \ind_{\acl_H(e) h_0}h$ and we get $\vec a \ind_{\acl_H(e) h_0}BH(M)$
so $\vec a \ind_{B h_0}H(M)$ and $h\subset h_0$.

\end{proof}

Finally, the following result on canonical bases can be proved using a minor modification of the
argument presented in \cite{BV-Tind}:

\begin{prop}\label{cbTind} Let $(M,H)$ be a sufficiently saturated $H$-structure of $T$,
$B\subset M$ an $H$-independent set, and $\vec a\in M$. Then
 $Cb_H(\vec a/B)$ and $Cb(\vec a HB(\vec a/B)/B))$ are interalgebraic.
\end{prop}

\begin{proof}
Let $e=Cb(\vec a HB(\vec a/B)/B))$. We saw in the previous lemma
that $\vec a\ind^{ind}_{e} B$ and thus $Cb_H(\vec a/B)\in
\acl_H^{eq}(e)$.

We will now prove that $e$ is in the algebraic closure of any Morley
sequence in $\stp_H(\vec a/B)$.

Let $\{\vec a_i: i<\omega\}$ be an
$\cL_H$-Morley sequence in $\tp_H(\vec a/\acl^{eq}_H(B))$.   Let
$h_j=HB(\vec a_j/B)$ (viewed as an imaginary representing a finite
set), so we have $h_j\in \dcl_H(\vec a_j B)$. Thus $\{\vec a_i h_i:
i<\omega\}$ is also an $\cL_H$-Morley sequence over $B$. 

\textbf{Claim} $\{\vec a_i h_i:
i<\omega\}$ is also an $\cL$-Morley sequence over $B$. 

Since $\{\vec a_i h_i: i<\omega\}$ is an $\cL_H$-Morley sequence over $B$, by Theorem  \ref{charforking}, we have
$h_j=HB(\vec a_j/B\vec a_{<j}h_{<j})$ and  $\vec a_j h_j \ind_{BH(M)}\vec a_{<j} h_{<j}H(M)$. Since $\vec a_j h_j \ind_{Bh_j}H(M)$
we get that $\vec a_j h_j \ind_{Bh_j}\vec a_{<j} h_{<j}$. But $h_j=HB(\vec a_j/B\vec a_{<j}h_{<j})$
and $B\vec a_{<j}h_{<j}$ is $H$-independent, so $B\vec a_{<j} h_{<j}\ind h_j$, and 
it follows that $\tp(\vec a_j h_j/B\vec a_{<j}h_{<j})$
does not fork (in the sense of $\cL$) over $B$. Thus, $\{\vec a_i h_i: i<\omega\}$ is a
$\cL$-Morley sequence in $\tp(\vec a h/B)$ over $B$. 

By supersimplicity of $T$, we know that $e$ is in the definable closure of an initial segment of $\{\vec a_i h_i: i<\omega\}$.
 
On the other hand, since $T^{ind}_{\mathcal{G}}$ is supersimple there is $N\in \omega$ such that
for all $n\geq N$, $\vec a_n\ind^{ind}_{\vec a_{<N}} B$. By Proposition \ref{aclp2}
$\acl_H(\vec a_{<N})$ is $H$-independent.
By Theorem \ref{charforking} and the fact that $\{\vec a_i : i<\omega\}$ is a Morley sequence in
$\tp_H(\vec a/\acl^{eq}_H(B))$, $HB(\vec a_n/B)=HB(\vec a_n/\acl_H(B\vec a_{<N}))=HB(\vec a_n/\acl_H(\vec a_{<N}))$ and in
particular $h_n\in \acl_H(\vec a_i : i<\omega)$ for every $n\geq N$. We then get
$e\in \acl^{eq}_H(\{\vec a_i : N\leq i<\omega\})$ as we wanted.

Next, since $\{\vec a_i : i<\omega\}$ is a Morley sequence in
$\tp_H(\vec a/\acl^{eq}_H(B))$, we have
$$\{\vec a_i : N\leq i<\omega\}\ind^{ind}_{Cb_H(\vec a/B)}B,$$
and thus also $$\{\vec a_i : N\leq i<\omega\}\ind^{ind}_{Cb_H(\vec
a/B)}e.$$ It follows that $e\in \acl_H^{eq}(Cb_H(\vec a/B))$, as
needed.
\end{proof}

\begin{prop}
Assume that now that $T$ is also superstable. Then $T_{\mathcal{G}}^{ind}$ is superstable.
\end{prop}

\begin{proof}
We already know that then $T_{\mathcal{G}}^{ind}$ is supersimple. To check stability we use
the criterion by Casanovas and Ziegler on expansions by predicates \cite{CZ}. We already 
know by Proposition \ref{induced} that the induced structure on $H$ is stable. We also know that
 $T_{\mathcal{G}}^{ind}$ is near model-complete, so every $\cl_H$-formula is equivalent to a 
 boolean combination of formulas of the form $\exists x_1\in H\dots \exists x_m\in H\varphi(\vec x,\vec y)$, where $\varphi(\vec x,\vec y)$ is an $\cL$-formula. We can conclude that the theory $T_{\mathcal{G}}^{ind}$ is 
 stable, thus also superstable.

\end{proof}

\section{One-basedness} \label{1based-section}

An example of a one-based geometric theory $T$ such that $T^{ind}$ is not one-based was given in 
 \cite{BeVa2}.  We follow the ideas on \cite{Ca} to understand exactly when one-basedness is preserved.

We will focus our study on two ``extreme'' cases, on one hand we will study the case when  $\mathcal{G}= ``x=x"$ in any theory $T$. On the other hand we will consider the case when $T$ is a theory of U-rank $\omega^{\alpha}$ and the U-rank is continuous (see Section \ref{omegacase}), in this case $\mathcal{G}$ is the union of the types of rank $\omega^{\alpha}$.  We will see that, in both cases, one-basedness is preverved in $T^{ind}_{\mathcal{G}}$ whenever $T$ satisfies some kind of ``triviality'' with respect to the types in $\mathcal{G}$.

\subsection{$\mathcal{G}= ``x=x''$}

In this subsection we assume $T$ is supersimple, let $\mathcal{G}= ``x=x''$ and assume that $T^{ind}_{\mathcal{G}}$ is first order.

\begin{defn}
A theory $T$ is {\it trivial} (or {\it trivial for freedom}), if for every set $A$ and every tuples $\vec a$, $\vec b$, and $\vec c$ pairwise independent over $A$, we have that $\vec a\ind_{A} \vec b \vec c$.
\end{defn}

\begin{rem}
The theory of the free pseudoplane (see example \ref{freeps}) is trivial.
\end{rem}

%It is also known that the theory of the pseudoplane eliminates quantifiers modulo the formulas $d_n(x,y)=\exists x_0...%\existsx_n(x=x_0Rx_1R...Rx_n=y)$. It is easy to check then that it has the nfcp.

\begin{lem}[\cite{Go}, Proposition 8]
Assume $T$ is one-based and trivial. If $\vec a\ind_{\vec a_1} \vec b_1$ and $\vec a\ind_{\vec a_2} \vec b_2$, then $\vec a\ind_{\vec a_1\vec a_2} \vec b_1\vec b_2$.
\end{lem}

\begin{prop}\label{prop1based}
Assume $T$ is a one-based theory and $\mathcal{G}=``x=x"$, then $T_{\mathcal{G}}^{ind}$ is one-based if and only if $T$ is trivial.

\end{prop}

\begin{proof}
($\Leftarrow$) Assume $T$ is trivial. Let $\vec a$ be a tuple, $B$ an algebraically closed set in $(M,H)$ and $\vec h=HB(\vec a)$. Since $\vec a\ind_{\vec h} H$ and $\vec a\ind_B B$, by the previous lemma we have that $\vec a\ind_{\vec h B}BH$. Therefore $HB(\vec a/B)\subset \vec h$ and 

\begin{eqnarray*}
\acl^{eq}_H(Cb_H(\vec a/B))&=&\acl^{eq}_H(Cb(\vec a HB(\vec a/B)/B))\\
&\subset & \acl^{eq}_H(Cb(\vec a \vec h/B))\quad \mbox{(because $HB(\vec a/B)\subset \vec h$ )}\\
&\subset & \acl^{eq}_H(\acl^{eq}(\vec a \vec h)\cap \acl^{eq}(B)) \quad \mbox{(because $T$ is one-based)}\\
&=& \acl^{eq}_H(\vec a)\cap \acl^{eq}_H(B).\\
\end{eqnarray*}

Thus, $T^{ind}_{\mathcal{G}}$ is one-based.

($\Rightarrow$) Assume $T^{ind}_{\mathcal{G}}$ is one-based and $T$ is not trivial. Then, there are real tuples $\vec a$, $\vec b$, $\vec h$ and a set of parameters $A$ such that $\vec a$, $\vec b$ and $\vec h$ are pairwise independent but not independent over $A$ (from now on we will work over $A$). By the generalized extension property we can assume that $\vec a \ind H$. Also, we may assume $b$ and $h$ are singletons and, as $h\ind \vec a$, we may assume $h$ belongs to $H$ by the density property.

We may also assume that $ b\ind_{\vec a h} H$ by the extension property.

Recall that $ b\nind_{\vec{a}} h$ and $h$ is a single element, then $h=HB( b/\vec a)$.

If we take $\vec a'=Cb( b h/\vec a)$ then the $\vec a'$, $ b$ and $h$ satisfy all the above properties, so we may take $\vec a=Cb( b h/\vec a)$.

By hypothesis $T^{ind}_{\mathcal{G}}$ is one-based, then $\acl^{eq}_H(Cb_H(b/\vec a))=\acl^{eq}_H( b)\cap \acl^{eq}_H(\vec a)$. Now, $\acl^{eq}_H(\vec a)=\acl^{eq}( \vec a)$ as $\vec a\ind H$. On the other hand, since $\vec a\ind H$  and $ b\ind_{h \vec a}H$, we have that $ b\underset{h}\ind H$. By hypothesis $ b\ind h$ and, by transitivity, $ b\ind H$. Hence $$HB( b)=\emptyset$$ and $$\acl^{eq}_H( b)=\acl^{eq}(b)$$ This means $\acl^{eq}_H(Cb_H( b/\vec a))=\acl^{eq}( b)\cap \acl^{eq}(\vec  a)$.

However, recall that $\vec a=Cb(bh/\vec a)$, therefore 

\begin{eqnarray*}
\acl^{eq}(\vec a)&=&\acl^{eq}_H(\vec a)\\
&=&\acl^{eq}_H(Cb(bh/\vec a))\\
&=&\acl^{eq}_H(Cb_H(b/\vec a))\\
&=&\acl^{eq}(b)\cap \acl^{eq}( \vec a).\\
\end{eqnarray*} 

This implies that $\vec a\subset \acl^{eq}(b)$, which yields a contradiction. 

\end{proof}

By a very well known result of Hrushovski, we now that if $T$ is stable, one-based and not trivial, then $T$ interprets an infinite group. Therefore we have the following corollary:

\begin{cor}
Assume $T$ is stable. Then $T^{ind}_{\mathcal{G}}$ is one-based if and only if $T$ is one-based and does not interpret an infinite group. 
\end{cor}

\subsection{Theories of rank $\omega^{\alpha}$} In  this subsection we will assume that $T$ is a theory of SU-rank $\omega^{\alpha}$ in which the $SU$-rank is continuous, we take $\mathcal{G}$ to be the union of all the types of rank $\omega^{\alpha}$ and we also assume that $T^{ind}_{\mathcal{G}}$ is first order.

We can define a closure operator as $\ccl(A)=\{x|SU(x/A)<\omega^{\alpha}\}$. It is not hard to check that $\ccl$ defines a pregeometry (monotonicity comes from the Lascar inequality).

\begin{defn}

A pregeometry $(X,cl)$ is {\it trivial} if for every $A\subset X$, $$cl(A)=\bigcup_{a\in A} cl(a).$$  
\end{defn}

Notice that if $G$ is a group of $SU$-rank $\omega^{\alpha}$ then $\ccl$ is not trivial (take $a\ind b$ both of rank $\omega$ and $c=a+b$, then $c\in \ccl(a,b)\setminus \ccl(a)\cup \ccl(b)$). 

\begin{rem}
In the theory of the free pseudoplane (see example \ref{freeps} ) the pregeometry generated by $\ccl$ is trivial: for $A$ algebraically closed and $a$ a single element, $SU(a/A)= d(a,A)$ where $d(a,A)$ is the minimum length of a path from $a$ to an element of $A$ (or $\omega$ if there is no path). If $b\in \ccl(A)$ it means that there is a path to some element $a\in A$ so $\ccl(A)=\bigcup_{a\in A}\ccl(a)$.
\end{rem}

%It is also known that the theory of the pseudoplane eliminates quantifiers modulo the formulas $d_n(x,y)=\exists x_0...%\existsx_n(x=x_0Rx_1R...Rx_n=y)$. It is easy to check then that it has the nfcp.

We will now prove that one-basedness is only preserved in $T^{ind}_{\mathcal{G}}$ when the pregeometry $\ccl$ is trivial. It is worth to notice that, unlike the $SU$-rank 1 case, the triviality of $\ccl$ does not imply that $T$ is one-based. In fact, the theory of the free pseudoplane is the canonical example of a CM-trivial theory which is not one-based. This is the reason why the statement of the following proposition is a little bit different from  the one from \cite{Ca}.

\begin{lem}
If $\ccl$ is trivial in $T$ then for every $\vec a$ and for every $B=\acl_H(B)$, $$HB(\vec a/B)\subset HB(\vec a).$$
\end{lem}

\begin{proof}
Let $h=HB(\vec a/B)=\{h_i|i\in I\}$. By minimality of H-bases, for every $i \in I$ we have $\vec a \nind_{Bh\setminus h_i} h_i$, then $h_i\in \ccl (\vec a Bh\setminus h_i)$.  As $B$ is $H$-independent and $h_i\notin B$ then $h_i\ind Bh\setminus h_i$, hence $h_i\notin \ccl (Bh\setminus h_i)$.  By triviality it means that $h_i\in \ccl(a_i)$ for some $a_i\in \vec a$. By exchange property $a_i\in \ccl(h_i)$, this implies $a_i\nind h_i$ and $a_i\ind_{h_i} H$ because $tp(a_i/h_i)$ is orthogonal to $H$, and $H$ consists of independent elements. We conclude that $h_i=HB(a_i)$ and $$HB(\vec a/B)=\{h_i|i\in I\}= \bigcup_{a_i\in A} HB(a_i)\subset HB(\vec a)$$.
\end{proof}

\begin{prop}\label{prop1-based}
Assume $T$ is one-based, then $T_{\mathcal{G}}^{ind}$ is one-based if and only if $\ccl$ is trivial in $T$.

\end{prop}

\begin{proof}%[Proof of Proposition \ref{prop1-based}]
($\Leftarrow$) Assume $\ccl$ is trivial, let $\vec a$ be a tuple, $B$ an algebraic closed set in $(M,H)$ and $\vec h=HB(\vec a/B)$.  By the characterization of canonical bases, $\acle_H(Cb_H(\vec a/B))=\acle_H(Cb(\vec a \vec h/B))$, as $T$ is one-based, $Cb(\vec a\vec h/B)\subset \acle(\vec a\vec h)$. By the previous lemma, $\vec h\subset HB(\vec a)$, and, thus,  $Cb_H(\vec a/B)\subset \acle_H(\vec a HB(\vec a))=\acle_H(\vec a)$, i.e. $T^{ind}_{\mathcal{G}}$ is one-based.

($\Rightarrow$) Assume $T^{ind}_{\mathcal{G}}$ is one-based and $\ccl$ is not trivial,  then there are a tuple $\vec a$ and elements $b$ and $h$ such that $b\in \ccl(\vec a h)$ and $b\notin \ccl(\vec a)\cup \ccl(h)$. Since $b\notin \ccl(\vec a)$ then $h\notin\ccl(\vec a)$, therefore $$SU(b)=SU(h)=\omega^{\alpha}.$$  We can take $\vec a$ a $\ccl$ independent tuple minimal with this property and, by the generalized extension property, we may assume that $\vec a\ind H$ . Moreover, as $h\notin \ccl(\vec a)$, we may assume also that $h$ belongs to $H$ by  the density property.

As $b \in \ccl(\vec a h)$, we have that $tp(b/\vec a h)$ is orthogonal to $H$, therefore $b\ind_{h\vec a}H$. Recall that $b\nind_{\vec a}h$ and $h$ is a single element, and, thus, $h=HB(b/\vec a)$.  By hypothesis $T^{ind}_{\mathcal{G}}$ is one-based, then $\acle_H(Cb_H(b/\vec a))=\acle_H(b)\cap \acle_H(\vec a)$. Now, $\acle_H(\vec a)=\acle(\vec a)$ as $\vec a\ind H$. On the other hand, as $\vec a\ind H$,  and $b\ind_{h \vec a}H$ we have $b\underset{h}\ind H$. By hypothesis $b\notin \ccl(h)$, hence $b\ind h$ (recall that $b$ is a single element) and by transitivity $b\ind H$. So $HB(b)=\emptyset$ and $\acle_H(b)=\acle(b)$. This means $\acle_H(Cb_H(b/\vec a))=\acle(b)\cap \acle(\vec a)$. But $\acle_H(Cb_H(b/\vec a))=\acle_H(Cb(bh/\vec a))$, so, in particular, we have that $$b\notin \ccl(Cb(bh/\vec a)).$$

Now, since $bh\ind_{Cb(bh/\vec a)}a,$
we have that $b\ind_{hCb(bh/\vec a)} h\vec a,$ therefore $$b\in \ccl(hCb(bh/\vec a)),$$ so any maximal $\ccl$-independent subset $\vec d$ of $\ccl(Cb(bh/\vec a))$ satisfies that $$b\in \ccl(\vec d h)$$ and $$b\notin \ccl(\vec d)\cup \ccl(h).$$
The minimality of the length of $\vec a$ yields $\ccl(Cb(bh/\vec a))=\ccl(\vec a)$, hence $\ccl(\vec a)=\ccl(\acle(\vec a)\cap \acle (b))\subset \ccl(\vec a)\cap \ccl(b)$, then $\vec a\in \ccl(b)$ and 
$h\in \ccl(\vec a b)\subset \ccl(b)$. This is a contradiction. 
\end{proof}

\section{Ampleness} \label{Ampleness-section}

The notion of ampleness, defined by Pillay, captures forking complexity. He proved in \cite{Pi} that a theory $T$ is one-based if and only if is not 1-ample, a theory $T$ is CM-trivial if and only if is not $2$-ample. Moreover if $T$ interprets a field then it is $n$-ample for every $n$.
We know that non-$1$-ampleness (one-basedness) is not always preserved in $T^{ind}_{\mathcal{G}}$. In contrast, we show in this section that, for $n\geq 2$, non-$n$-ampleness is preserved in the expansion.

Note that, in contrast with the previous section, we do not require any conditions on $\mathcal{G}$.

\begin{defn}
A supersimple theory $T$ is {\it $n$-ample} if  (possibly after naming some parameters) there exist tuples $a_0,...,a_n$ in $M$  satisfying the following conditions: \\

For all $1\leq i\leq n-1$.\\

(1) $a_{i+1}\underset{a_i}\ind a_{i-1}...a_0$,

(2) $ \acl^{eq} (a_0...a_{i-1}a_{i+1})\cap \acl^{eq}(a_0...a_{i-1}a_{i})=\acl^{eq} (a_{0}...a_{i-1})$.
\\

(3) $a_n\underset{\acl^{eq}(a_1)\cap \acl^{eq}(a_0)}\nind a_0$.

The tuple $a_0,...,a_n$ is $n$-ample if it satisfies the above properties.

\end{defn}

Following \cite{Ca} we prove that, for $n\geq 2$, non-$n$-ampleness is preserved in $T^{ind}_{\mathcal{G}}$. First we need the following lemmas.

\begin{lem}[\cite{Pi1} Fact 2.4 ]
If $A$, $B$, $C$ and $D$ are algebraically closed such that $A\cap B=C$,  and $AB\ind_C D$, then $\acl^{eq}(AD)\cap \acl^{eq}(BD)=\acl^{eq}(CD)$.
\end{lem}

\begin{proof}
Let $c\in \acl^{eq}(AD)\cap \acl^{eq}(BD)$. Since $D\ind_C AB$ then $AD\ind_A AB$ and $cD\ind_A AB$. Therefore $$Cb(cD/AB)\subset A.$$ In the same way we obtain that $$Cb(cD/AB)\subset B.$$ Therefore  $$Cb(cD/AB)\subset A\cap B=C,$$ so $$cD\ind_{C} AB.$$  This implies that $c\in \acl^{eq}(DC)$.
\end{proof}

\begin{lem}
A theory is $n$-ample if and only if  there exist tuples $a_0,...,a_n$ in $M^{eq}$  satisfying the following conditions:

(1) $a_{i+1}\underset{a_i}\ind a_{i-1}...a_0$ for all $ i \leq n-1$,

(2) $ \acl^{eq} (a_{i-1}a_{i+1})\cap \acl^{eq}(a_{i-1}a_{i})=\acl^{eq} (a_{i-1})$ for all $i\leq n-2$.
\\

(3) $a_n\underset{\acl^{eq}(a_1)\cap \acl^{eq}(a_0)}\nind a_0$.

\end{lem}

\begin{proof}
($\Rightarrow$) Assume $T$ is n-ample and let $a_0,...,a_n$ be the ample tuple. Then $$\acl^{eq} (a_{i-1}a_{i+1})\cap \acl^{eq}(a_{i-1}a_{i})\subset \acl^{eq} (a_0,...,a_{i-1})\cap \acl^{eq}(a_{i-1}a_{i}).$$ Since $a_i\ind_{a_{i-1}} a_0...a_{i-2}$, this intersection is contained in $\acl^{eq}(a_{i-1})$.

($\Leftarrow$) Assume $a_0,...,a_n$ satisfy the above conditions, then using the previous lemma with $A=\acl^{eq} (a_{i-1}a_{i+1})$, $B=\acl^{eq}(a_{i-1}a_{i})$, $C=\acl^{eq} (a_{i-1})$ and \newline $D=\acl^{eq}(a_0,...,a_{i-2})$, we get that  $$\acl^{eq} (a_0...a_{i-1}a_{i+1})\cap \acl^{eq}(a_0...a_{i-1}a_{i})=\acl^{eq} (a_{0}...a_{i-1}).$$ i.e. $a_0,...,a_n$ is $n$-ample.
\end{proof}

We will use the above characterization of ampleness to prove the main result of this section (Theorem \ref{amplethm}). 

\begin{lem}\label{Hbas}
Let $A\subset B$, $A=\acl^{eq}_H(A)$ and $B=\acl^{eq}_H(B)$. If $\acl^{eq}_H(cA)\cap B= A$ then $HB(c/A)\subset HB(c/B)$.
\end{lem}

\begin{proof}
By Lemma \ref{changinghb} we have $$HB(c/A)\subset HB(c/B)\cup HB(B).$$ 

Now, if $HB(c/A)\cap H(B) = \emptyset$ we are done, but $$HB(c/A)\cap HB(B)\subset \textstyle \acl^{eq}_H(cA)\cap B= A$$ and $HB(c/A)\cap A =\emptyset$.
\end{proof}

\begin{thm}\label{amplethm}
For $n\geq 2$, the theory $T$ is n-ample if and only if $T^{ind}_{\mathcal{G}}$ is.
\end{thm}

Before doing the proof, recall that the canonical bases in $T^{ind}_{\mathcal{G}}$ are interalgebraic with canonical bases of $T$, therefore every ``new'' imaginary element is interalgebraic with an ``old'' imaginary. By virtue of Lemma \ref{aclp} we may assume that $\acl^{eq}_H(A)=\acle(A HB(A))$.

\begin{proof}

Assume that $T$ is $n$-ample. Let $a_0,...,a_n$ be such that
\begin{enumerate}

\item $\acl^{eq}(a_ia_{i+1})\cap \acl^{eq}(a_ia_{i+2})=\acl^{eq}(a_i)$  for all $i<n-1$,

\item $a_{i+1}\ind_{a_i}a_{0}...a_{i-1}$ for all $i<n$,

\item $a_n\nind_{\acl^{eq}(a_1)\cap \acl^{eq}(a_0)} a_0$.

 \end{enumerate}
By {\it the generalized extension property}, we may assume that $a_0...a_n\ind H$, hence $HB(X)=\emptyset$ for every $X\subset \{a_0,...,a_n\}$, which implies that $\acl^{eq}(X)=\acl^{eq}_H(X)$ for $X\subset \{a_0,...,a_n\}$. 
Therefore

\begin{eqnarray*}
\acl^{eq}_H(a_ia_{i+1})\cap \acl^{eq}_H(a_ia_{i+2})&=& \acl^{eq}(a_ia_{i+1})\cap \acl^{eq}(a_ia_{i+2})\\
&=&\acl^{eq}(a_i)\\
&=&\acl^{eq}_H(a_i).\\
\end{eqnarray*}

On the other hand, we have $$a_{i+1}\ind^{ind}_{a_i}a_0...a_{i-1},$$ because $$a_{i+1}\ind_{a_i}a_0...a_{i-1}$$ and $HB(a_{i+1}/a_0...a_i)=\emptyset$.

Finally, since $a_n\nind_{\acl^{eq}(a_1)\cap \acl^{eq}(a_0)} a_0$, we have that $a_n\nind^{ind}_{\acl^{eq}(a_1)\cap \acl^{eq}(a_0)} a_0$. Thus, the tuple $a_0,...a_n$ is n-ample in the sense of $T^{ind}_{\mathcal{G}}$.

Assume now that $T$ is not $n$-ample, let $a_0,...,a_n \in \mathfrak{C}$ be tuples such that:

\begin{enumerate}

\item $\acl^{eq}_H(a_ia_{i+1})\cap \acl^{eq}_H(a_ia_{i+2})=\acl^{eq}_H(a_i)$  for all $i<n$.

\item $a_{i+1}\ind^{ind}_{a_i}a_{0}...a_{i-1}$ for all $i<n$.

 \end{enumerate}

We may assume also that $\acl^{eq}_H(a_i)=\acl^{eq}(a_i)$ for every $i$.

Consider $h_i= HB(a_n/a_i)$. We claim that $h_0\subset h_1\subset \cdots \subset h_{n-1}$.

\begin{proof}[Proof of the claim]
Since $a_n\ind^{ind}_{a_{i+2}} a_ia_{i+1}$, we have that $a_na_{i}\ind^{ind}_{a_{i}a_{i+2}}a_ia_{i+1}$. Hence

\begin{eqnarray*}
 \acl^{eq}_H(a_na_{i})\cap  \acl^{eq}_H(a_{i}a_{i+1})&\subset & \acl^{eq}_H(a_{i}a_{i+2})\cap  \acl^{eq}_H(a_{i}a_{i+1})\\
 &\subset & \acl^{eq}_H(a_i)\mbox{ by assumption (1)}\\
 \end{eqnarray*}

Using Lemma \ref{Hbas} (by making $A=a_i$, $B=a_ia_{i+1}$ and $c=a_n$) we conclude that $h_i=HB(c/A)\subset HB(c/B)=HB(a_n/a_ia_{i+1})$. But $a_n\ind^{ind}_{a_{i+1}}a_i$, hence $HB(a_n/a_ia_{i+1})=HB(a_n/a_{i+1})=h_{i+1}$.\end{proof}

In particular $h_0\subset h_{n-1}$. 

Notice that $$h_{0}\subset h_{n-2}\subset \acl^{eq}_H(a_{n-2}a_n),$$ hence $$\acl^{eq}(a_{n-2}a_n h_0)\subset \acl_H^{eq} (a_{n-2}a_n)$$ and

\begin{eqnarray*}
\acl^{eq}(a_{n-2}a_{n-1})\cap \acl^{eq}(a_{n-2}a_nh_0)&\subset &\acl^{eq}_H(a_{n-2}a_{n-1})\cap \acl^{eq}_H(a_{n-2}a_n)\\
&=&\acle_H(a_{n-2}).
\end{eqnarray*}

On the other hand, $$a_nh_{n-1}\ind^{ind}_{a_{n-1}}a_0...a_{n-1},$$ so since $a_n a_{n-1} h_{n-1}$ is $H$-independent

$$a_nh_{n-1}\ind_{a_{n-1}}a_0...a_{n-1},$$ and

$$a_nh_{0}\ind_{a_{n-1}}a_0...a_{n-1} \mbox{ (because }h_0\subset h_{n-1}),$$

then, setting $a_n'=a_nh_{0}$ and $a'_{i}=a_i$ for $i< n$, we have that:

\begin{enumerate}

\item $\acl^{eq}(a'_ia'_{i+1})\cap \acl^{eq}(a'_ia'_{i+2})=\acl^{eq}(a'_i)$  for all $i<n$. 

\item $a'_{i+1}\ind_{a'_i}a'_{0}...a'_{i-1}$ for all $i<n$.

 \end{enumerate}

Since $T$ is not $n$-ample, it follows that $$a'_n\ind_{\acl^{eq}(a'_0)\cap\acl^{eq}(a'_1)}a_0'$$ i.e. $$a_nh_{0}\ind_{\acl^{eq}(a_0)\cap\acl^{eq}(a_1)}a_0.$$

As $h_0=HB(a_n/a_0)$, by the characterization of canonical bases we have that $Cb(a_nh_0/a_0)$ is interalgebraic with  $Cb_H(a_n/a_0)$; this implies that $$a_n\ind^{ind}_{\acl^{eq}_H(a_0)\cap\acl^{eq}_H(a_1)}a_0.$$ Thus, $T^{ind}_{\mathcal{G}}$ is not $n$-ample.
\end{proof}

\section{Geometry modulo $H$ in the one-based case}\label{geometrymodH} In this
section we consider the case when $T$ is one-based of SU-rank $\omega^\alpha$ in which SU-rank is continuous, and follow the
proofs of Theorem 5.13 \cite{Va1} and the results of Section 6 of \cite{Va1}, and Section 4 of  \cite{BeVa2}, to study
the geometry induced by $\ccl$ localized at $H(M)$. Many of the proofs are nearly identical to the ones from \cite{Va1} and \cite{BeVa2}, we include them for completeness.
As before, we take ${\mathcal{G}}$ to be a type of SU-rank $\omega^\alpha$.

Let $(M,H)$ be
a sufficiently saturated  model of $T^{ind}_{\mathcal{G}}$. Let $\ccl_H$ be the
localization of the operator $\ccl$ at $H(M)$, i.e.
$\ccl_H(A)=\ccl(A\cup H(M))$. Thus, $a\in\ccl_H(B)$ means $SU(a/B\cup
H(M))<\omega^\alpha$.

 \begin{prop}\label{geom1based} Suppose $T$ is
one-based. Then the pregeometry $(M,\ccl_H)$ is modular. \end{prop}

\begin{proof} It suffices to show that for any $a,b\in M$ and a
small set $C\subset M$, if $a\in\ccl_H(bC)$ then there exists $d\in
\ccl_H(C)$ such that $a\in\ccl_H(bd)$. We may assume that
$a,b\not\in\ccl_H(C)$. Thus, $SU(a/CH(M))=SU(b/CH(M))=\omega^\alpha$.

Let $\vec h\in H(M)$ be finite such that
$a\in\ccl(bC\vec h)$. Let $e=Cb(ab/C\vec h)$. Thus, by
one-basedness of $T$, $e\in\acl^{eq}(ab)\cap\acl^{eq}(C\vec h)$.
By the density property, there is  $b'\models
\tp(b/\acl^{eq}(C\vec h))$, $b'\in H(M)$. Take $a'\in M$ such that
$\tp(a'b'/\acl^{eq}(C\vec h))=\tp(ab/\acl^{eq}(C\vec h))$. Then
$e\in\acl^{eq}(a'b')$. Clearly, $a'\in\ccl(b'C\vec h)\subset
\ccl_H(C)$. Also, $ab\ind_e C\vec h$ implies $SU(a/be)=SU(a/bC\vec
h)<\omega^\alpha$. Since $e\in\acl^{eq}(a'b')$, we have $SU(a/ba'b')\le
SU(a/be)<\omega^\alpha$. Since $b'\in H(M)$, this implies
$a\in\ccl_H(ba')$. Hence, taking  $d=a'$, we have $d\in \ccl_H(C)$
and $a\in\ccl_H(bd)$, as needed.\end{proof}

Let $(M^*, \ccl^*)$ be the geometry associated with $(M,\ccl_H)$ (i.e. $M^*$ is the set $M\backslash \ccl_H(\emptyset)$ modulo the relation $\ccl_H(x)=\ccl_H(y)$) . For any $a\not\in \ccl_H(\emptyset)$, let $a^*$ be
 the class of $a$ modulo the relation $\ccl_H(x)=\ccl_H(y)$. Define the relation $\sim $ by $$a^*\sim b^*\iff |\ccl^*(a^*,b^*)|\ge 3\ {\rm\ or\ }  \  a^*=b^*.$$

\begin{lem}\label{thirdpoint}
For any $a,b\in M$, $a^*\sim b^*$ if and only if there exist $d_1,\ldots,d_{n}\in M$ such that $$a^*\in\ccl^*(b^*d_1^*\ldots d_n^*)\backslash\ccl^*(d_1^*\ldots d_n^*).$$
\end{lem}

\begin{proof}
The "only if" direction is clear. For the "if" direction, suppose $a^*\neq b^*$ and $a^*\in\ccl^*(b^*d_1^*\ldots d_n^*)\backslash\ccl^*(d_1^*\ldots d_n^*)$. We may assume that $n\ge 1$ is minimal such. 
Then $a\in \ccl(bd_1\ldots d_nh_1\ldots h_k)$ for some $h_1,\ldots, h_k\in H(M)$. We may assume that $k$ is minimal such. Then the tuple $abd_2\ldots d_nh_1\ldots h_k$ is $\ccl$-independent.
By the density property, we can find $d_2',\ldots, d_n'\in H(M)$ such that $\tp(d_2',\ldots, d_n'/ab\vec h)=\tp(d_2,\ldots, d_n/ab\vec h)$. Let $d_1'\in M$ be such that 
$$\tp(d_1',d_2',\ldots, d_n'/ab\vec h)=\tp(d_1,d_2,\ldots, d_n/ab\vec h).$$ Then $d_1'\not\in\ccl_H(\emptyset)$ and  $(d_1')^*\in\ccl^*(a^*,b^*)$, while $(d_1')^*\neq a^*, b^*$. 
Thus, $|\ccl^*(a^*,b^*)|\ge 3$, as needed.

\end{proof}

\begin{lem}\label{eqrel}
 The relation $\sim$ is an equivalence on $M^*$.
\end{lem}

\begin{proof}

 Reflexivity and symmetry are clear. For transitivity,  assume
$a^* \sim b^* \sim c^*$, with all three distinct. Then there exist
$d_1^* \in \ccl^*(a^* b^*)\backslash \lbrace a^*, b^*\rbrace$ and
$d_2^* \in \ccl^*(b^*, c^*)\backslash \lbrace b^*, c^*\rbrace$. If
$d_1^* = d_2^*$, then $c^*\in \ccl^*(b^*, d_2^*) = \ccl^*(b^*, d_1^*) =
\ccl^*(a^*, d_1^*)$, and therefore $d_1^* = d_2^* \in \ccl^*(a^*, c^*)
\backslash \lbrace a^*, c^*\rbrace$, hence $a^*\sim c^*$.

Now, assume that $d_1^*\neq d_2^*$ and $a^*\in \ccl^*(d_1^*, d_2^*)$.
If $a^* = d_2^*$, then $b^*$ witnesses $a^*\sim c^*$.  If $a^*\neq
d_2^*$, then $d_2^*\in \ccl^*(a^*, d_1^*)$. We also have $b^*\in
\ccl^*(a^*, d_1^*)$, $c^*\in \ccl^*(b^*, d_2^*)$.  Thus, $c^* \in \ccl^*(a^*,
d_1^*)$.  If $c^* = d_1^*$, $b^*$ witnesses $a^*\sim c^*$.  If
$c^*\neq d_1^*$, then $d_1^*$ witnesses $a^*\sim c^*$. Finally,
assume that $d_1^*\neq d_2^*$ and neither $a^*\not\in \ccl^*(d_1^*,
d_2^*)$. Then $$a^*\in \ccl^*(c^* d_1 d_2^*) \backslash \ccl^*(d_1^*
d_2^*).$$  Thus, by Lemma \ref{thirdpoint}, $a^*\sim c^*$.

\end{proof}

For  any $a^*\in M^*$ let $[a^*]$ denote the $\sim$-class of $a^*$.\\

\begin{lem}\label{decom1}  The $\sim$-classes are closed in the sense of $\ccl^*$, i.e. for any
$a^*\in M^*$, we have $\ccl^*([a^*]) = [a^*]$.
\end{lem}

\begin{proof} Assume $c^*\in \ccl^*(b_1^*, \ldots, b_n^*)$, $\vec b^*=
(b_1^*, \ldots, b_n^*) \in [a^*]$
minimal such tuple, and $n>1$ (if $n = 1$, we have $c^*=b_1^*$).
Then $b_1^*\ldots b_{n-1}^*$ witnesses $c^*\sim b_n^*$, by Lemma
\ref{thirdpoint}.
\end{proof}

For any geometry $(X, Cl)$, a  
non-empty subset
 of $X$, with the closure operator induced by $Cl$, is referred to as a {\it subgeometry} of $(X,Cl)$. Clearly, a subgeometry is itself a geometry.  Next lemma shows that
$\sim$ splits $(M^*, \ccl^*)$ into disjoint subgeometries of the form
$([a^*], \ccl^*)$, with no "interaction" between them.

\begin{lem}\label{decom2} For any $A\subset M^*$, $\ccl^*(A) = \bigcup_{[a^*]\in
M^*/\sim}\ccl^*(A\cap [a^*])$.
\end{lem}

\begin{proof}
Suppose $c^*\in\ccl^*(A)$, and  $a_1^*, \ldots, a_n^*\in A$ is a tuple such that $c\in \ccl^*(a_1^*, \ldots, a_n^*)$, and $n$ is minimal such. 
It suffices to show that $a_i^*$ all come from the same $\sim$-class.
If $n=1$, we are done.  Suppose $n > 1$.  Then $c^* a_3^* \ldots
a_n^*$ witnesses $a_1^*\sim a_2^*$ by Lemma \ref{thirdpoint}. Similarly,  $a_1^*\sim a_i^*$ for all $2 < i \le n$. Thus, all $a_i^*$ belong to the same $\sim$-class.
\end{proof}

Next, we will show that the $\sim$-classes are
either singletons or infinite dimensional (as geometries).

\begin{lem}\label{decom3} If $|[a^*]| > 1$, then $\dim([a^*])$ is infinite.
\end{lem}

\begin{proof}
Suppose there exists $b^*\sim a^*$, $b^*\neq a^*$. 
Let $c^*\in \ccl^*(a^*,b^*)\backslash\{a^*,b^*\}$. 
Let $a, b, c\in M$
be representatives of the classes $a^*, b^*$ and $c^*$ modulo the relation $\ccl_H(x)=\ccl_H(y)$,
respectively.

 Then $SU(a/H(M))=SU(b/aH(M))=\omega$. By the extension property,  we can find $b_i\models \tp(b/a)$, $i\in \omega$,
independent over $aH(M)$. Then, by Lemma \ref{H-indtypes},  $\tp_H(b_i/a)=\tp_H(b/a)$ for all $i\in\omega$. 
Also, $b_i$ are $\ccl_H$-independent over $a$. Let $c_i$ be such that $\tp_H(b_ic_i/a)=\tp_H(bc/a)$ for $i\in\omega$.
Passing to the geometry $(M^*, \ccl^*)$, we get $b_i^*\sim a^*$ witnessed by $c_i^*$,
$i\in\omega$, with $b_i$ $\ccl^*$-independent over $a^*$. This shows that that
$([a^*], \ccl^*)$ is
 infinite dimensional.
\end{proof}

Recall the following classical fact (see \cite{Hal})
about projective geometries.

\begin{fact}\label{projective}
 A non-trivial modular geometry of
 dimension $\ge 4$ in which any closed set of dimension $2$ has
 size $\ge 3$ is a projective geometry over some division ring.
\end{fact}

\begin{lem}\label{decom4} If $T$ is one-based and $|[a^*]| > 1$ , the geometry $([a^*], cl)$
is an infinite dimensional projective geometry over some division
ring.
\end{lem}

\begin{proof} By Proposition \ref{geom1based}, $(M^*,\ccl^*)$ is modular.  By Lemma  \ref{decom2}, $[a^*]$ is closed in $(M^*,\ccl^*)$, and hence $([a^*],\ccl^*)$ is also modular. 
Since $|[a^*]|>1$, $([a^*],\ccl^*)$  is non-trivial (there are two distinct point having a third one in its closure).
Thus, the statement follows by Fact \ref{projective} and the definition of $\sim$.
\end{proof}

We are now ready to prove the characterization of the geometry of
$\ccl_H$, as well as the original
geometry of $\ccl$ in the case when $T$ is one-based.

\begin{prop}\label{geom_1b}
Suppose $T$ is a  one-based supersimple theory of SU-rank $\omega$,
$(N,H)$ a sufficiently (e.g. $|T|^+$-) saturated 
models of $T^{ind}_{\mathcal{G}}$, and $M$ a small model of $T$ (e.g. of size $|T|$).
Then
\\
(1) The geometry $(N^*,\ccl^*)$ of $\ccl_H$ in $(N,H)$ is a
disjoint union of infinite dimensional projective geometries over
division rings and/or a trivial geometry.
\\
(2) The geometry of the closure operator $\ccl$ in $M$ is a disjoint union
of subgeometries of projective geometries over division rings.
\end{prop}

\begin{proof}
(1)  Follows by Lemmas \ref{decom2}, \ref{decom3} and \ref{decom4}.
\\
(2) By Lemma \ref{existence}, any structure of the form
$(M,H)$ where $M\models T$, and $H(M)$ is an independent set of generics,
can be embedded, in an $H$-independent way, in a sufficiently
saturated $H$-structure. Thus we may assume that
$(M,\emptyset)\subset (N,H)$ with $M\ind_\emptyset H(N)$.
Then $\ccl$-independence over $\emptyset$ in $M$ coincides with
$\ccl$-independence in $N$ over $H(N)$, i.e. $\ccl_H$-independence.
Thus, we have a natural embedding of the associated geometry of
$(M,\ccl)$ into $(N^*, \ccl^*)$. The result now follows by (1).
\end{proof}

\begin{rema}\label{geom_1b-modularity}
The previous proposition also holds with the weaker assumption that the pregeometry
$(N,\cl_H)$ is modular instead of asking that $T$ is one-based. All the proofs depend
on the properties of the closure operator, not the properties of forking in the full 
structure.
\end{rema}

\end{document}